\documentclass[12pt]{article}
\usepackage{amsthm}
\usepackage{amsmath}
\usepackage{amsfonts}
\usepackage{amssymb}
\usepackage{url}
\usepackage{graphicx}
\usepackage{color}
\usepackage{psfrag}
\usepackage{subfigure}

\newtheorem{thm}{Theorem}[section]
\newtheorem{cor}[thm]{Corollary}

\newtheorem{prop}[thm]{Proposition}

\theoremstyle{definition}
\newtheorem{example}[thm]{Example}
\newtheorem{remark}[thm]{Remark}

\makeatletter  
\g@addto@macro\@verbatim\small 
\makeatother 

\newcommand{\R}{\mathbb{R}}
\newcommand{\C}{\mathbb{C}}

\newcommand{\N}{\mathbb{N}}
\newcommand{\Q}{\mathbb{Q}}
\newcommand{\PP}{\mathbb{P}}

\newcommand{\SSS}{\mathcal{S}}

\DeclareMathOperator{\conv}{conv}

\DeclareMathOperator{\Span}{Span}

\DeclareMathOperator{\THbody}{TH}

\numberwithin{equation}{section}

\date{ }

\title{\bf Dualities in \\ Convex Algebraic Geometry}

\author{Philipp Rostalski and Bernd Sturmfels}

\begin{document}

\maketitle

\begin{abstract}
Convex algebraic geometry concerns the interplay between 
optimization theory and real algebraic geometry. Its objects of study include convex 
semialgebraic sets that arise in semidefinite programming and from sums of squares.
This article compares three notions of duality that are relevant in these
contexts: duality of convex bodies, duality of projective varieties, 
and the Karush-Kuhn-Tucker conditions derived from Lagrange duality. We show
that the optimal value of a polynomial program is an algebraic 
function whose minimal polynomial is expressed by the hypersurface 
projectively dual to the constraint set. We give an exposition of recent 
results on the boundary structure of the convex hull of a compact variety,
we contrast this to Lasserre's representation as a spectrahedral shadow,
and we explore the geometric underpinnings of semidefinite programming duality.
 \end{abstract}


\begin{figure}
 \centering
\includegraphics[width=0.55\textwidth]{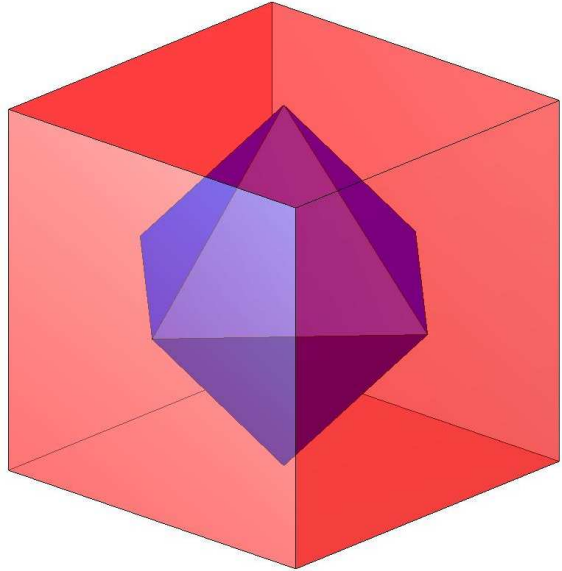}
\caption{The cube is dual to the octahedron.\label{fig::cube_and_dual}}
\end{figure}

\section{Introduction}
Dualities are ubiquitous in mathematics and its applications. 
This article compares several
notions of duality that are relevant for the interplay between
convexity, optimization,  and algebraic geometry.
It is primarily expository, and is
intended for a diverse audience, ranging from 
graduate students in mathematics to practitioners of
optimization who are based in engineering.

Duality for vector spaces lies at the heart of linear algebra and
functional analysis. Duality in convex geometry is an involution on the set of convex bodies: for instance, it 
maps the cube to the octahedron and vice versa (Figure \ref{fig::cube_and_dual}).
Duality in optimization, known as {\em Lagrange duality},
plays a key role in designing efficient algorithms for the solution
of various optimization problems. In projective geometry,
points are dual to hyperplanes, and this leads to a natural
notion of {\em projective duality} for algebraic varieties.

Our aim here is to explore these dualities and their interconnections
in the context of polynomial optimization and semidefinite programming.
Towards the end of the Introduction, we shall discuss the context and organization
of this paper. At this point, however, we jump right in and present
a concrete three-dimensional example that illustrates our perspective
on these topics.

\subsection{How to Dualize a Pillow}

We consider the following symmetric matrix with three indeterminate entries:
\begin{equation}
\label{eq:matrixQ}
 Q(x,y,z) \quad = \qquad  \begin{pmatrix} 1 & x & 0 & x \\
               x & 1 & y & 0 \\
                0 & y & 1 & z \\
                x & 0 & z & 1 \end{pmatrix}.
\end{equation}
This symmetric $4 {\times} 4$-matrix specifies a $3$-dimensional compact convex body  
\begin{equation}
\label{eq::simpleSpectrahedron}
P \,\, = \,\, \bigl\{ \,(x,y,z) \in \R^3 \,\,|\, \, Q(x,y,z) \,\succeq\, 0 \, \bigr\}.
\end{equation}
 The notation ``$\succeq 0$''  means that the matrix is {\em positive semidefinite},
i.e.,~all four eigenvalues are non-negative real numbers.  Such a {\em linear matrix inequality}
always defines a closed convex set which is 
referred to as a {\em spectrahedron}.

\begin{figure}
\vskip -0.4cm \!\!\!
\includegraphics[width=0.55\textwidth]{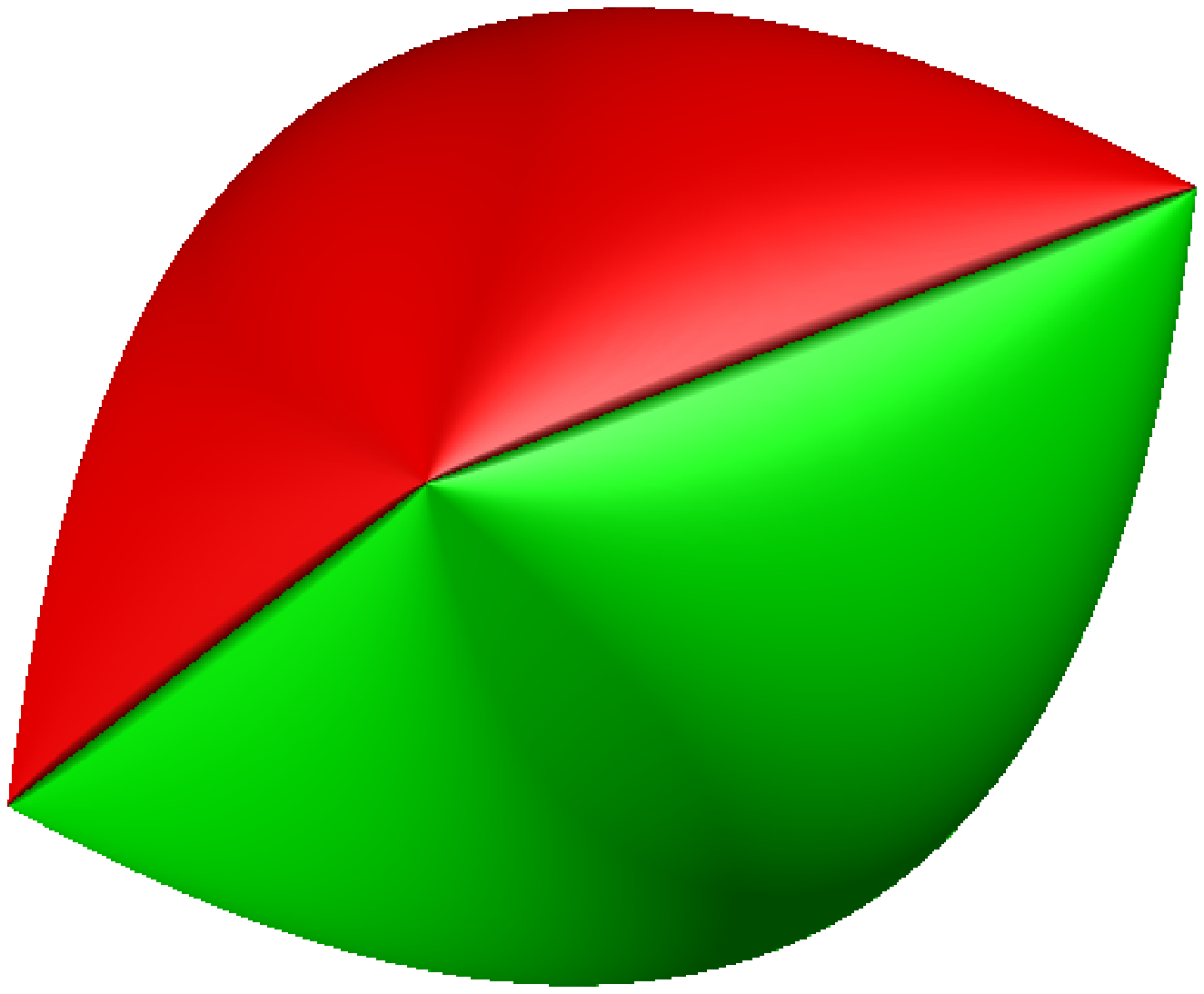} \!\!\!\!\!
\includegraphics[width=0.46\textwidth]{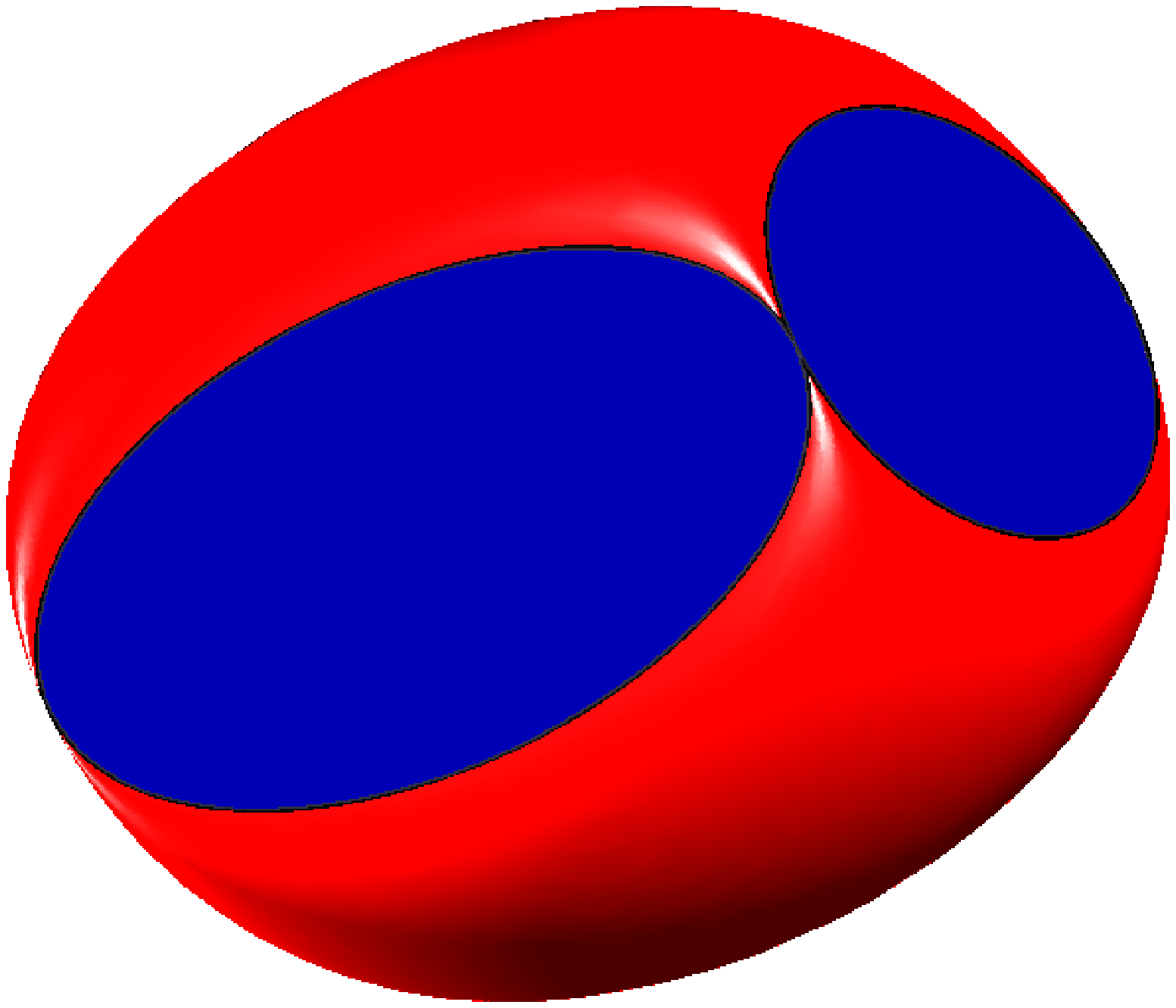}
\caption{A 3-dimensional spectrahedron $P$ and its dual convex body $P^\Delta$.\label{fig:honest_pillow}}
\end{figure}

Our spectrahedron $P$ looks like a pillow. It is shown on the left in
Figure~\ref{fig:honest_pillow}.
The {\em algebraic boundary} of $P$ is the surface specified by the determinant 
\[
\det(Q(x,y,z)) \quad = \quad
x^2(y-z)^2
  -2 x^2-y^2-z^2+1
\quad = \quad 0.
\]

The interior of $P$  represents all matrices $Q(x,y,z)$ 
whose four eigenvalues are positive.  At all smooth points on the boundary of $P$,
precisely one
eigenvalue vanishes, and the rank of the matrix $Q(x,y,z)$ drops from $4$ to $3$.
However, the rank drops further to $2$ at the four singular points
\[
(x,y,z) \,\,\,= \,\,\, 
\frac{1}{\sqrt{2}} (1,1,-1) ,\,\,
\frac{1}{\sqrt{2}} (-1,-1,1) , \,\, 
\frac{1}{\sqrt{2}} (1,-1,1) ,\,\,
\frac{1}{\sqrt{2}} (-1,1,-1).
\]
We find these from a {\em Gr\"obner basis}   of the ideal 
of $3 \times 3$-minors of $Q(x,y,z)$:
\[
\bigl\{\, 2 x^2 - 1,\, 2 z^2-1 ,\, y+z \, \bigr\} .
\]
The linear polynomial $y+z$ in this Gr\"obner basis
defines the symmetry plane of the pillow $P$.
The four corners form a square in that plane. Its edges are also edges
of $P$. All other faces of $P$ are exposed points. These come in two 
families, called {\em protrusions}, one 
above the plane $y+z = 0 $ and one below it.

Like all convex bodies, our pillow $P$ has an associated {\em dual convex body}
\begin{equation}
\label{eq::dualsimpleSpectrahedron}
P^\Delta \,\, = \,\, \left\{\, (a,b,c) \in \R^3 \,\,|\,\, ax + by + cz \leq 1\,\,\,
\hbox{for all} \,\,(x,y,z) \in P \, \right\},
\end{equation}
consisting of all linear forms that evaluate to at most one on $P$.
Our notation $P^\Delta$ is chosen to be consistent with that in Ziegler's text book
\cite[\S 2.3]{Zie}.

The dual pillow $P^\Delta$ is shown on the right in Figure~\ref{fig:honest_pillow}.
Note the association of faces under duality. The pillow
$P$ has four $1$-dimensional faces, four singular $0$-dimensional faces, and two 
smooth families of 
$0$-dimensional faces. The corresponding dual faces of $P^\Delta$ have dimensions $0$, $2$ and $0$  respectively.

{\em Semidefinite programming} is the computational problem of minimizing
a linear function over a spectrahedron.
For our pillow $P$, this takes the form
\begin{equation}
\label{eq::simpleSDPp}
\begin{aligned}
p^*(a,b,c)\,\,=& \quad \underset{(x,y,z)\in \R^3}{\text{Maximize}} & & a x + b y + c z \\
& \text{subject to}
& & Q(x,y,z) \succeq 0.
\end{aligned}
\end{equation}
We regard this as a parametric optimization problem:
we are interested in the optimal value and optimal solution of
(\ref{eq::simpleSDPp}) as a function of $(a,b,c) \in \R^3$.
This function can be expressed in terms of the dual body $P^\Delta$ as follows:
\begin{equation}
\label{eq::simpleSDPlambda}
\begin{aligned}
p^*(a,b,c)\,\,=& \quad \underset{\lambda \in\R }{\text{Minimize}} & & \lambda\\
& \text{subject to}
& & \frac{1}{\lambda} \cdot (a,b,c) \,\in\, P^\Delta.
\end{aligned}
\end{equation}

We distinguish this formulation from the duality in
semidefinite programming. The
{\em dual} to (\ref{eq::simpleSDPp}) is the following program with $7$ 
decision variables:
\begin{equation}
\label{eq::simpleSDPd}
\begin{aligned}
d^*(a,b,c)\,\,=&\quad \underset{u \in \R^7}{\text{Minimize}} & & u_1 + u_4 + u_6 + u_7 \\
& \text{subject to}
& & \!\!\!\! \begin{pmatrix}
 2 u_1 & 2 u_2 & & u_3 & - 2 u_2 {-} a \, \\
 2 u_2 & 2 u_4 & & -b & 2 u_5 \\
 2 u_3 & - b & & 2 u_6 & - c \\
 -2 u_2 {-} a & 2 u_5 & & - c & 2 u_7 
\end{pmatrix} \, \succeq \, 0.
\end{aligned}
\end{equation}
Since (\ref{eq::simpleSDPp}) and (\ref{eq::simpleSDPd}) are both strictly feasible, strong duality holds \cite[\S 5.2.3]{BV}, i.e. the two programs attain the same optimal value: $\,p^*(a,b,c)=d^*(a,b,c)$. 
Hence, problem~\eqref{eq::simpleSDPd} can be derived from \eqref{eq::simpleSDPlambda}, as we 
shall see in Section~\ref{sec::sdp}.

We write $M(u; a,b,c)$ for the $4 {\times} 4$-matrix in (\ref{eq::simpleSDPd}).
 The following equations and inequalities, known as the
{\em Karush-Kuhn-Tucker conditions} (KKT), are necessary and sufficient
for any pair of optimal solutions:
\begin{align*}
Q(x,y,z) \cdot M(u; a,b,c)& \, = \,0,\qquad (\text{complementary slackness}) \\
Q(x,y,z) & \, \succeq \, 0,\\
M(u;a,b,c) & \, \succeq  \, 0.
\end{align*} 
We relax the inequality constraints
and consider the system of  equations
\[
\lambda = ax + by + cz  \quad \hbox{and} \quad Q(x,y,z) \cdot M(u; a,b,c) \, = \,0. 
\]
This is a system of $11$ equations. Using computer algebra,
we eliminate the $10$ unknowns $\,x,y,z, u_1,\ldots,u_7$.
The result is a polynomial in $a,b,c$ and~$\lambda$.
Its factors, shown in (\ref{eq:lambdadual1})-(\ref{eq:lambdadual2}),
 express the optimal value $\lambda^*$  in terms of~$a,b,c$.
 
At the optimal solution, the product of the two
$4 {\times} 4$-matrices $Q(x,y,z)$ and $M(u; a,b,c)$ is zero,                   
and their respective ranks are either $(3,1)$ or $(2,2)$. 
In the former case the optimal value $\lambda^*$ is one of the two solutions of                                  
 \begin{equation}
 \label{eq:lambdadual1} \!
 (b^2+2bc+c^2) \cdot \lambda^2-a^2b^2-a^2c^2-b^4-2b^2c^2-2bc^3-c^4-2b^3c \,\,=\,\, 0.                                                                 
 \end{equation}
 In the latter case it comes from the four corners of the pillow, and it satisfies
 \begin{equation}
 \label{eq:lambdadual2} 
 \begin{matrix} 
& (2 \lambda^2-a^2+2 a b-b^2+2 b c-c^2-2 a c)  &
\\ \cdot \!\!\!\! & (2 \lambda^2-a^2-2ab-b^2+2bc-c^2+2ac) & \,\,=  \quad 0.
 \end{matrix}
\end{equation}
These two equations  describe the algebraic boundary of the dual body $P^\Delta$. 
Namely, after setting $\lambda = 1$, the irreducible polynomial in (\ref{eq:lambdadual1})
describes the  quartic surface that makes up the curved part of the
boundary of $P^\Delta$, as seen in  Figure \ref{fig:honest_pillow}. In addition,
there are four planes spanned by flat $2$-dimensional faces of $P^\Delta$.
The product of the four corresponding affine-linear forms equals (\ref{eq:lambdadual2}).
Indeed, each of the two quadrics in (\ref{eq:lambdadual2}) factors into two linear factors.
These two characterize the planes spanned by opposite $2$-faces of $P^\Delta$.

The two equations (\ref{eq:lambdadual1}) and (\ref{eq:lambdadual2}) also
offer a first glimpse at the concept of {\em projective duality} in algebraic
geometry. Namely, consider the surface in projective
space $\PP^3$ defined by ${\rm det}(Q(x,y,z)) = 0$ after replacing the ones along the diagonals by a homogenization variable.
Then  (\ref{eq:lambdadual1}) is its dual surface in the dual projective space
$(\PP^3)^*$. The surface (\ref{eq:lambdadual2}) in $(\PP^3)^*$ is dual to the
$0$-dimensional variety in $\PP^3$ cut out by the $3 {\times} 3$-minors of $Q(x,y,z)$.

The optimal value function of the optimization problem  (\ref{eq::simpleSDPp})
is given by the algebraic surfaces dual to the boundary of $P$ and its singular locus.
We have seen two different ways of dualizing (\ref{eq::simpleSDPp}):
the dual optimization problem (\ref{eq::simpleSDPd}), and 
the optimization problem (\ref{eq::simpleSDPlambda}) on $P^\Delta$.
These two formulations are related as follows. If we regard
(\ref{eq::simpleSDPd}) as specifying a $10$-dimensional spectrahedron,
then the dual pillow $P^\Delta$ is a projection of that spectrahedron:
\[
P^\Delta \,\, = \,\,
\bigl\{(a,b,c) \in \R^3 \,| \,
\exists u \in \R^7 \,: \, M(u;a,b,c) \succeq 0 \,\,\text{and} \,\,
u_1 + u_4 + u_6 + u_7 \,= \,1 \bigr\}.
\]
Linear projections of spectrahedra are called
{\em spectrahedral shadows}. These objects play a prominent
role in the interplay between semidefinite programming and  convex algebraic geometry.
The dual body to a spectrahedron is generally not a spectrahedron, but it
is always a spectrahedral shadow.

\subsection{Context and Outline}

Duality is a central concept in convexity and convex optimization,
and numerous authors have written about their connections and their interplay
with other notions of duality and polarity. Relevent references include
 Barvinok's text book \cite[\S 4]{Bar} and the survey  by
Luenberger \cite{Lue}. The latter focuses on dualities used in engineering,
such as duality of vector spaces, polytopes, graphs, 
 and  control systems.
 The objective of this article is to revisit the theme of 
 duality in the context of {\em convex algebraic geometry}. This emerging
 field aims to exploit algebraic structure in convex optimization problems,
 specifically in semidefinite programming and polynomial optimization.
In algebraic geometry, there is a natural notion of projective duality,
which associates to every algebraic variety a dual variety. One of our goals is to explore the meaning of projective
duality for optimization theory.

Our presentation is organized as follows. In Section~\ref{sec::ingredients} we cover
 preliminaries needed for the rest of the paper. Here
the various dualities are carefully defined and their basic properties
are illustrated by means of examples. In Section~\ref{sec::opt_value} we derive the
result that the optimal value function of a polynomial program is
represented by the defining equation of the hypersurface 
projectively dual to the manifold describing the boundary of all feasible solutions.
This highlights the important fact that the duality best
known to algebraic geometers
arises very naturally in convex optimization.
Section~\ref{sec::convex_hull} concerns the convex hull of a compact algebraic variety
in $\R^n$. We discuss recent work of Ranestad and Sturmfels
\cite{RS1, RS2} on the hypersurfaces in the boundary of
such a convex body, and we present several new examples
and applications. 

In Section~\ref{sec::sdp} we focus on semidefinite programming (SDP), and we offer
a concise geometric introduction to SDP duality.
This leads to the concept of algebraic degree of SDP
\cite{BR, NRS}, or, geometrically, to projective duality for 
varieties defined by rank constraints on symmetric matrices of linear forms.

A {\em spectrahedral shadow} is the image of a spectrahedron under 
a linear projection. Its dual body is a linear section
of  the dual body to the spectrahedron. In Section~\ref{sec::spec_shadow} we examine
this situation in the context of sums-of-squares programming, and
we discuss linear families of non-negative polynomials.

\section{Ingredients}
\label{sec::ingredients}
In this section we review the mathematical preliminaries needed for the
rest of the paper, we give precise definitions,
and we fix more of the notation. We begin with the notion
of duality for vector spaces and cones therein, then move on to
convex bodies,  polytopes,  Lagrange duality in optimization,
the KKT conditions, projective duality in algebraic geometry,
   and discriminants.

\subsection{Vector Spaces and Cones}

We fix an ordered field $K$. The primary example is the field of real numbers, $K = \R$, 
but it makes much sense to also allow other fields,
such as the rational numbers $K = \Q\,$ or the real Puiseux series $K = \R \{ \! \{ \epsilon \} \! \}$.
For a finite dimensional $K$-vector space $V$, the {\rm dual vector space} is the set $V^*={\rm Hom}(V,K)$ of all linear forms on $V$. Let $V$ and $W$ be vector spaces and $\varphi: V \rightarrow W$ a linear map. The \emph{dual map} $\varphi^*: W^* \rightarrow V^*$ is
the linear map defined by $\varphi^*(w) = w \circ \varphi \in V^*$ for every $w\in W^*$. 
If we fix bases of $V$ and $W$ then $\varphi$ is represented by a matrix $A$.
The dual map $\varphi^*$ is represented, relative to the dual bases for $W^*$ and $V^*$,
 by the {\em transpose} $A^t$ of the matrix~$A$.
 
A subset $C\subset V$ is a {\em cone} if it is closed under
multiplication with positive scalars. A cone $C$ need not be convex, but its {\em dual cone}
\begin{equation}
\label{eq:dual_cone}
C^* \,\, = \,\, \left\{ \,l \in V^* \,\,|\,\,\forall x\in C \,:\, l(x)\geq 0 \,\, \right\}
\end{equation}
is always closed and convex in $V^*$. If $C$ is a  \emph{convex cone} 
then the second dual $(C^*)^*$ is the closure of $C$.
 Thus, if $C$ is a closed convex cone in $V$ then
\begin{equation}
\label{eq:conebiduality}
   (C^*)^* \,\, = \,\, C . 
   \end{equation}
This important relationship is referred to as  \emph{biduality}. 

Every linear subspace $L \subset V$ is also a cone. Its dual cone
is the orthogonal complement of the subspace:
\[
L^* \,\, = \,\, L^\perp \,\,= \,\, \left\{\, l \in V^*\, \,|\,\, \forall x \in L \,: \, l(x)= 0 \; \right\}.
\]
The dual map to the inclusion $L \subset V$ is the projection
$\pi_L: V^* \rightarrow V^*/L^\perp$.
 Given any cone $ C \subset V$, the intersection $C \cap L$ is a cone in $L$.
 Its dual cone $(C\cap L)^*$  is the  projection of the cone $C$ into 
 $V^*/L^\perp$. More precisely,
 \[
 (C\cap L)^* \,=\, C^*+ L^\perp
\quad \hbox{in} \,\,\, V^*.
\]
Now, it makes sense to consider this convex set modulo $L^\perp$. We thus obtain
\begin{equation}
 \label{eq:intersectproject}
(C \cap L)^* \,\,=  \,\,\overline{\pi_L(C^*)}
\quad\, \hbox{in} \,\,\,  V^*/L^\perp.
\end{equation}
This identity shows that projection and intersections are dual operations.

A subset $F\subseteq C$ of a convex set $C$ is a {\em face} 
if $F$ is itself convex and contains any line segment $L\subset C$ whose relative interior intersects $F$. We say that $F$ is an {\em exposed face} if there exists a linear functional $l$ that attains its minimum over $C$ 
precisely at $F$.  Clearly, an exposed face is a face, but the converse does not hold.
For instance, the edges of the red triangle in Figure~\ref{fig::trigonometric} are
non-exposed faces of the $3$-dimensional convex body shown there.

An exposed face $F$ of a cone $C$ 
determines a face of the dual cone $C^*$ via
\[
F^\diamond \,\,= \,\, \{ \, l \in C^* \,\,|\,\, l \text{ attains its minimum over $C$ at $F$ }\}\,.
\]
The dimensions of the faces $F$ of $C$ and $F^\diamond$ of $C^*$ satisfy
the inequality
\begin{equation}
\label{eq:dimformula}
 {\rm dim}(F) \, + \,{\rm dim} (F^\diamond) \,\,\leq \,\, {\rm dim}(V). 
 \end{equation}
 If $C$ is a polyhedral cone then $C^*$ is also polyhedral.
 In that case, the number of faces $F$ and $F^\diamond$ is finite and 
 equality holds in (\ref{eq:dimformula}). On the other hand,
 most cones considered in this article are not polyhedral, they have
 infinitely many faces, and the inequality in (\ref{eq:dimformula}) is
 usually strict.  For instance, the {\em second order cone}
 $\,C\, = \, \{ \,(x,y,z) \in \R^3 \,:\, \sqrt{x^2 + y^2}  \leq z \} \,$
 is self-dual,  each proper face $F$ of $C$ is $1$-dimensional,
 and  the formula (\ref{eq:dimformula}) says  $1 + 1 \leq 3$.
 
\subsection{Convex Bodies and their Algebraic Boundary}
\label{subsec::convbody}
A {\em convex body} in $V$ is a full-dimensional convex set that is closed and bounded. If $C$
is a cone and $z \in  {\rm int} (C^*)$ then  $C\cap \{z=1\}$ is a convex body
in the hyperplane $\{\, z = 1 \,\}$ of $V$. In this manner, every pointed
$d$-dimensional cone gives rise to a $(d{-}1)$-dimensional convex body,
and vice versa. These transformations, known as {\em homogenization} 
and {\em dehomogenization}, respect faces and algebraic boundaries.
 They allow us to go back and fourth between convex bodies
and  cones in the next higher dimension. For instance, the $3$-dimensional body $P$
 in (\ref{eq::simpleSpectrahedron})
corresponds to the cone in $\R^4$ we get by multiplying the constants
$1$ on the diagonal in (\ref{eq:matrixQ}) with a new variable.

Let $P$  be a full-dimensional convex body in $V$ and assume that $0 \in {\rm int} (P)$.
Dehomogenizing the definition for cones, we obtain the {\em dual convex body} 
\begin{equation}
\label{eq:Pdelta}
P^\Delta \,\,= \,\, \left\{ \,\ell \in V^* \,|\, \; \forall x\in P \,:\, \ell(x)\leq 1\, \right\}.
\end{equation}
This is derived from \eqref{eq:dual_cone}
using the identification $l(x)=z-\ell(x)$ for $z=1$.

Just as in the case of convex cones, if $P$ is closed then biduality holds:
\[ (P^\Delta)^\Delta \,\, = \,\, P .\]
The definition (\ref{eq:Pdelta}) makes sense for arbitrary subsets $P$  of $V$.
That is, $P$ need not be convex or closed.
A standard fact from convex analysis \cite[Cor.~12.1.1 and \S14]{Rock}
says that the double dual is the closure of the convex~hull:
\[
(P^\Delta)^\Delta \,\,\, = \,\,\, \overline{\conv(P\cup {0})}.
\]

All convex bodies discussed in this article are {\em semialgebraic},
that is, they can be described by polynomial inequalities. 
We note that
if $P$ is semialgebraic then its dual body $P^\Delta$ is also semialgebraic.
This is a consequence of
Tarski's theorem on quantifier elimination in real algebraic geometry \cite{BPR, BCR}.

The {\em algebraic boundary} of a semialgebraic convex body $P$, denoted
$\partial_a P$, is the  smallest algebraic variety that contains the boundary $\partial P$.
In geometric language, $\partial_a P$ is the  {\em Zariski closure} of $\partial P$.
It is identified with the squarefree polynomial $f_P$ that vanishes 
on $\partial P$. Namely,
 $\,\partial_a P  = V(f_P)\,$ is the zero set of the polynomial $f_P$.
 Note that $f_P$ is unique up to a multiplicative constant.
 Thus $\partial_a P$ is an algebraic hypersurface which
 contains the boundary  $\partial P$.
 
A {\em polytope} is the convex hull of a finite subset of $V$.
If $P$ is a polytope then so is its dual $P^\Delta$ \cite{Zie}. The boundary of $P$
consists of finitely many {\em facets} $F$. These are the 
faces $F = v^\diamond$ dual to the {\em vertices} $v$ of $P^\Delta$.
The algebraic boundary $\partial_a P$ is the arrangement of
hyperplanes spanned by the facets of $P$. Its defining polynomial
$f_P$ is the product of the linear polynomials $v-1$.

\begin{example}\label{ex::poly1}
A polytope known to everyone is the three-dimensional cube
\[
P \,\,=\,\, \conv\{(\pm 1,\pm 1, \pm 1)\} \,\,= \,\,\{ -1 \leq x,y,z\leq 1\}.
\]
 Figure~\ref{fig::cube_and_dual} illustrates the familiar fact that
its dual polytope is the octahedron
\[
P^\Delta \,\,= \,\, \{-1 \leq a\pm b\pm c\leq 1\} 
\,\,= \,\, \conv \{ \pm e_1, \pm e_2, \pm e_3 \}.\]
Here $e_i$ denotes the $i$th unit vector.
The eight vertices of $P$ correspond to the facets of $P^\Delta$,
and the six facets of $P$ correspond to the vertices of $P^\Delta$.
The algebraic boundary of the cube is described by a degree $6$ polynomial
\[
\partial_a P \,\, = \,\,V\left((x^2-1)(y^2-1)(z^2-1)\right).
\]
The algebraic boundary of the octahedron is given by a degree $8$ polynomial
\[
\partial_a P^\Delta \,\,= \,\, V\left(\prod (1 - a\pm y\pm c) \prod(a\pm b\pm c+1)\right).
\]
Note that $P$ and $P^\Delta$ are the unit balls for the norms $L_\infty$ and $L_1$ on $\R^3$. \qed
\end{example}

Recall that the {\em $L_p$-norm} on $\R^n$ is defined by
$\,\| x \|_p \,= \,(\sum_{i=1}^n |x_i|^p)^{1/p} \,$ for $x \in \R^n$. The {\em dual norm}
to the $L_p$-norm
is the $L_q$-norm  for $\frac{1}{p}+\frac{1}{q} = 1$, that is,
\[
\|y\|_q \,\,= \,\, {\rm sup} \{\langle y,x\rangle \,|\, x \in \R^n,\, \|x\|_p \leq 1 \}.
\]
Geometrically,  the unit balls for these norms are dual as convex~bodies.
 
\begin{figure}
\!\!\!\!\!
\includegraphics[width=0.56\textwidth]{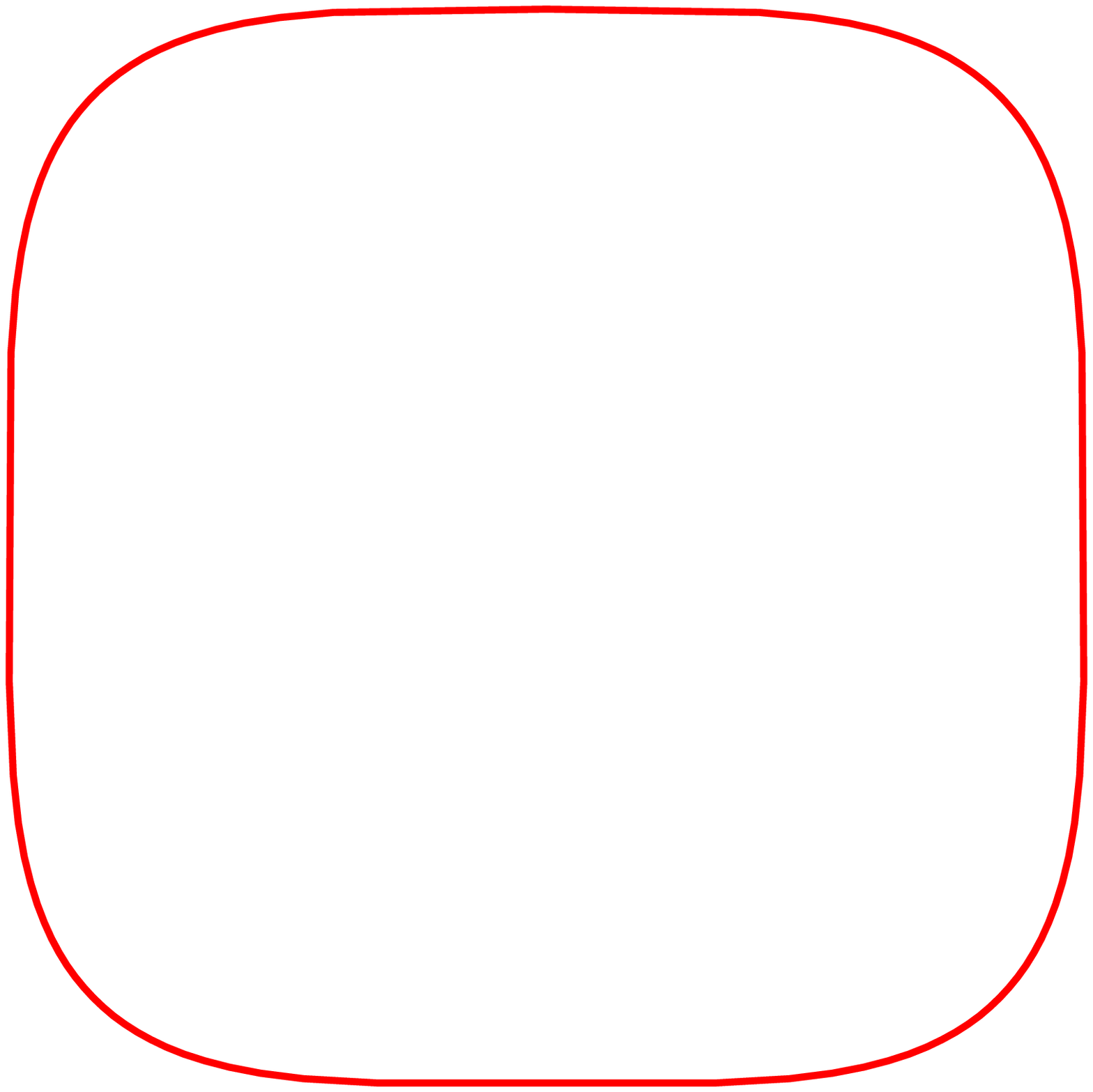} \!\!\!\!\!\!\!\!\!
\includegraphics[width=0.5\textwidth]{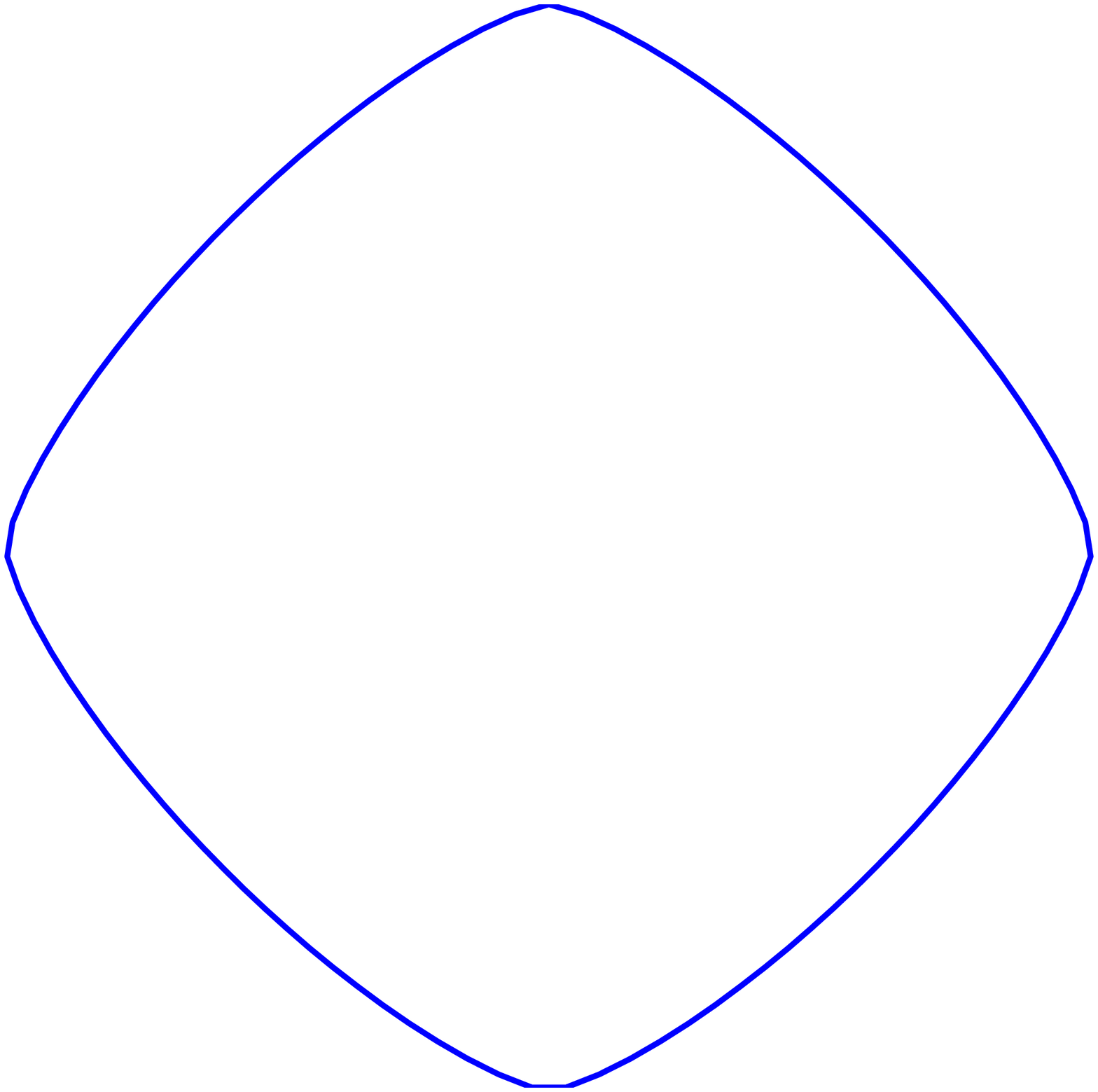}
\vskip -0.34cm
\caption{The unit balls for the $L_4$ norm and the $L_{4/3}$ norm are dual.
The curve on the left has degree $4$, while its dual curve on the right has degree~$12$.}
\label{fig::norm}
\end{figure}

\begin{example} 
\label{ex:fourtwelve} 
 Consider the case $n=2$ and $p=4$. Here the unit ball equals
  \[ P \,\,= \,\,\{ \,(x,y) \in \R^2 \,: x^4+y^4 \leq 1\,\} . \]
 This planar convex set is shown in Figure~\ref{fig::norm}.
 In this example, since the curve is convex,
 the ordinary boundary coincides with the algebraic boundary,
 $\, \partial_a P \, = \, \partial P $,
 and is represented by the defining quartic polynomial $\,  x^4+y^4 -1$.
 
The dual body is the unit ball for the $L_{4/3}$-norm on $\R^2$:
\[
P^\Delta \,\,= \,\, \{ (a,b) \in \R^2 \,:\, |a|^{4/3}+|b|^{4/3} \leq 1\}\,.
\]
The algebraic boundary of $P^\Delta$ is an irreducible algebraic curve of degree $12$,
\begin{equation}
\label{eq:curvedeg12}
\partial_a P^\Delta \,= \,
 V \! \left(a^{12}{+}3 a^8 b^4{+}3 a^4 b^8{+}b^{12}{-}3 a^8 {+}21 a^4 b^4 {-} 3 b^8
 {+}3 a^4 {+}3 b^4 {-}1\right) \! , 
 \end{equation}
which again coincides precisely with the (geometric) boundary $\partial P^\Delta$.
This dual polynomial is easily produced by the following one-line program in 
the computer algebra system \ {\tt Macaulay2} \ due to Grayson and Stillman
 \cite{M2}:
\begin{verbatim}
R = QQ[x,y,u,v]; eliminate({x,y},ideal(x^4+y^4-1,x^3-u,y^3-v))
\end{verbatim}
In Subsection 2.4 we shall introduce the
algebraic framework for
performing such duality computations, not just for curves, but for arbitrary varieties.
\qed
\end{example}

\subsection{Lagrange Duality in Optimization}

We now come to a standard concept of duality in optimization theory. Let us
consider the  following general nonlinear polynomial optimization problem:
\begin{equation}
\label{eq::nonlinearOPT}
\begin{aligned}
& \underset{x}{\text{minimize}} & &f(x) \\
& \text{subject to}
& & g_i(x)\leq 0, \quad i=1,\ldots,m,\\
& & & h_j(x)= 0, \quad j=1,\ldots,p.
\end{aligned}
\end{equation}
Here the $g_1,\ldots,g_m, h_1,\ldots,h_p$ and $f$ are
polynomials in $\R[x_1,\ldots,x_n]$.
The {\em Lagrangian} associated to the optimization problem (\ref{eq::nonlinearOPT}) is
the function
\[ \begin{matrix}
& L\,: \, \R^n \times \R_+^m \times \R^p & \rightarrow &  \R^n \qquad \qquad \qquad  \quad
\qquad \qquad \qquad \\ & \qquad
(x,\lambda,\mu) & \mapsto & f(x)+\sum_{i=1}^{m} \lambda_i g_i(x) +\sum_{j=1}^{p} \mu_j h_j(x)
\end{matrix}
\]
The scalars $\lambda_i\in \R_+$ and $\mu_j \in \R$ are the {\em Lagrange multipliers} for the constraints $g_i(x) \leq 0$ and $h_j(x)=0$. The Lagrange function $L(x,\lambda,\mu)$ can be interpreted as an augmented cost function with penalty terms for the constraints.
For more information on the above formulation see  \cite[\S 5.1]{BV}.

One can show that problem~\eqref{eq::nonlinearOPT} is equivalent to finding
\begin{equation*}
u^* \,\,= \,\, \underset{x \in \R^n}{\text{minimize}}\; \underset{\mu \in \R^p \,{\rm and}\, \lambda\geq 0}{\text{maximize}}\quad L(x,\lambda,\mu).
\end{equation*}
The key observation here is
that any positive evaluation of one of the polynomials
$g_i(x)$, or any non-zero evaluation of one of the polynomials $h_j(x)$,
would render the inner optimization problem unbounded.

The {\em dual optimization problem} to \eqref{eq::nonlinearOPT} is
obtained by exchanging the order of the two nested optimization subproblems
in the above formulation:
\begin{equation*}
\begin{aligned}
v^*\,\,=\,\, \underset{\mu  \in \R^p \,{\rm and}\, \lambda\geq 0}{\text{maximize}}\;\underbrace{\underset{x \in \R^n}{\text{minimize}}\quad L(x,\lambda,\mu)}_{\phi(\lambda,\mu)}.
\end{aligned}
\end{equation*}
The function $\phi (\lambda,\mu)$ is known as the
 \emph{Lagrange dual function} to our problem.
This function is always concave, so the dual is
always a convex optimization problem.
It follows from the definition of the dual function that $\phi(\lambda,\mu) \leq u^*$ for all $\lambda,\mu$. Hence
the optimal values satisfy the inequality
\[ 
v^* \,\,\leq \,\, u^*. 
\]
If equality happens, $v^* = u^*$, then we say that {\em strong duality} holds.
A necessary condition for strong duality is  $\,\lambda_i^* g_i(x^*)=0$ for all $i=1,\ldots,m$, where $x^*,\lambda^*$ denote the primal and dual optimizer. We see this by inspecting the Lagrangian and taking into account the fact that $h_j(x)=0$ for all feasible~$x$.

Collecting all inequality and equality constraints in the primal and dual optimization problems yields the following  optimality conditions:
\begin{thm}[Karush-Kuhn-Tucker (KKT) conditions]
Let $(x^*,\lambda^*,\mu^*)$ be primal and dual optimal solutions with $u^*=v^*$ (strong duality). Then
\begin{align}
\label{eq::KKT}
\!\!\! \nabla_x f\Big|_{x^*} + \,\sum_{i=1}^{m} \lambda_i^*  \cdot \nabla_x g_i\Big|_{x^*}
 + \,\sum_{j=1}^{p} \mu_j^* \cdot \nabla_x  h_j \Big|_{x^*}&\,= \,\, 0,\nonumber\\
g_i(x^*)&\,\leq \,\,0 \quad {\rm for} \,\, i=1,\ldots,m, \nonumber\\
\lambda_i^* &\,\geq \,\,0 \quad {\rm for} \,\, i=1,\ldots,m, \nonumber\\
h_j(x^*)&\, = \,\, 0 \quad {\rm for} \,\, j=1,\ldots,p, \\
\text{Complementary slackness:} \quad \lambda^*_i \cdot g_i(x^*)&\, = \,\, 0
 \quad {\rm for} \,i=1,\ldots,m. \nonumber
\end{align}
\end{thm}

For a derivation of this theorem see  \cite[\S 5.5.2]{BV}.
Several comments on the KKT conditions are in order. First, we note that complementary 
slackness amounts to a case distinction between active ($g_i=0$) and inactive inequalities ($g_i<0$). 
For any index $i$ with $g_i(x^*)\ne 0$ we need $\lambda_i = 0$, so the corresponding inequality does not play a role in the gradient condition. On the other hand, if $g_i(x^*) = 0$, 
then this can be treated as an equality constraint.

From an algebraic point of view, it is natural to relax the inequalities and to 
focus on the {\em KKT equations}.  These are the polynomial equations in (\ref{eq::KKT}):
\begin{equation}
\label{eq:KKTeqns}
h_1(x) \,=\, \cdots \,=\, h_p(x) \, = \,\, \lambda_1 g_1(x)
\, = \, \lambda_m g_m(x)
 \,\,= \,\, 0 .
\end{equation}
If we wish to solve our optimization problem exactly then
we must compute the algebraic variety 
in  $\,\R^n \times \R^m \times \R^p\,$ that is defined by these equations.

In what follows we explore Lagrange duality and the KKT conditions in two special cases, namely
in optimizing a linear function over an algebraic variety (Section~\ref{sec::opt_value}) and in semidefinite programming (Section~\ref{sec::sdp}).

\subsection{Projective varieties and their duality}
\label{subsec::proj_dualities}

In algebraic geometry, it is customary to work over an algebraically
closed field, such as the complex numbers $\C$.
 All our varieties will be defined over a subfield $K$ of the real numbers $\R$, 
and their points have coordinates in $\C$.
It is also customary to work in projective space $\PP^n$ rather
than affine space $\C^n$, i.e., we work with equivalence 
classes $x \sim \lambda x$ for all $\lambda \in \C \backslash \{0\}$,
 $x\in \C^{n+1} \backslash \{0\}$. Points 
$(x_0:x_1:\dots:x_n) $ in projective space $\PP^n$ are lines through the origin
in $\C^{n+1}$, and the usual affine coordinates are obtained by \emph{dehomogenization} 
e.g. with respect to $x_0$ (i.e. setting $x_0=1$). All points with $x_0=0$ are 
then considered as points at infinity. 
We refer to \cite[Chapter 8]{CLO} for an elementary introduction to projective algebraic geometry. 

Let $I = \langle h_1,\ldots,h_p \rangle$  be a homogeneous ideal in the polynomial ring $K[x_0,x_1,\ldots,x_n]$. We write $X = V(I)$ for its variety in the projective space $\PP^n$ over $\C$. The {\em singular locus} $\,{\rm Sing}(X)\,$ is a proper subvariety of $X$. It is defined inside $X$ by the vanishing of the $c \times c$-minors of the $m {\times} (n{+}1)$-Jacobian matrix
$J(X) = \bigl(\partial h_i / \partial x_j \bigr)$, where  $c = {\rm codim}(X)$.
See \cite[\S 9.6]{CLO} for background 
on singularities and dimension.
While the matrix $J(X)$ depends on our choice of ideal generators $h_i$,
the singular locus of $X$ is independent of that choice.
Points in ${\rm Sing}(X)$ are called {\em singular points} of $X$.
We write  $\,X_{\rm reg} = X \backslash {\rm Sing}(X)\,$ for the set of
{\em regular points} in $X$. We say that the projective variety $X$ 
is smooth if ${\rm Sing}(X) = \emptyset$, or equivalently, if $X = X_{\rm reg}$.

The dual projective space $(\PP^n)^*$ parametrizes
hyperplanes in $\PP^n$. A point $(u_0:\cdots :u_n) \in (\PP^n)^*$
represents the hyperplane
$\bigl\{ x \in \PP^n \,|\, \sum_{i=0}^n u_i x_i = 0 \bigr\}$.
We say that $u$ is {\em tangent} to $X$ at a regular point $ x \in X_{\rm reg}$ 
if $x$ lies in that hyperplane and its representing
vector $(u_1,\ldots,u_n)$ lies in the row space
of the Jacobian matrix $J(X)$ at the point~$x$. 

We define the {\em conormal variety} ${\rm CN}(X)$ of $X$ to be the closure
of the set
\[
\bigl\{ (x,u) \in \PP^n \times (\PP^n)^* \,\,|\,\, x \in X_{\rm reg} \,\,
\text{and $u$ is tangent to $X$ at $x$} \,\bigr\}.
\]
The projection of ${\rm CN}(X)$ onto the second factor is denoted $X^*$ and is
called the dual variety. More precisely,
the {\em dual variety} $X^*$ is the closure of the set
\[
\bigl\{ \,u \in (\PP^n)^*
\,\,|\,\,\,\text{the hyperplane $u$ is tangent to $X$ at some regular point} \, \bigr\}.
\]

\begin{prop} \label{prop:CNdim}
The conormal variety ${\rm CN}(X)$ has dimension $n-1$.
\end{prop}

\begin{proof}
We may assume that $X$ is irreducible. Let $c = {\rm codim}(X)$. 
There are $n{-}c$ degrees of freedom in picking a point $x$ in $X_{\rm reg}$.
Once the regular point $x$ is fixed, the possible tangent vectors $u$ to $X$ at $x$
form a linear space of dimension $c{-}1$. Hence the dimension of ${\rm CN}(X)$ is
$(n{-}c) + (c{-}1) = n{-}1$.
\end{proof}

Since the dual variety $X^*$ is a linear projection of the conormal variety ${\rm CN}(X)$,
Proposition \ref{prop:CNdim} implies that
the dimension of $X^*$ is at most is $n-1$. We typically expect $X^*$ to have
dimension $n-1$, i.e.~regardless of the dimension of $X$,
the dual variety $X^*$ is typically a hypersurface in $(\PP^n)^*$.

\begin{example}[Example~\ref{ex:fourtwelve} cont.]
\label{ex:fourtwelve2}
Fix coordinates $(x {:} y {:} z)$ on $\PP^2$ and consider the
ideal $I = \langle x^4+y^4 - z^4 \rangle$. Then
$X = V(I)$ is the projectivization of the quartic curve
in Example \ref{ex:fourtwelve}. The dual curve $X^*$ is the
projectivization of the curve $\partial_a P^\Delta$  in \eqref{eq:curvedeg12}. Hence,
$X^*$ is a curve of degree $12$ in $(\PP^2)^*$. \qed
\end{example}

To compute the dual $X^*$ of a given variety $X$, we can utilize Gr\"obner bases
\cite{CLO, M2} as follows.
We augment the ideal $I$ with the bilinear polynomial
$\sum_{i=0}^n u_i x_i $ and all the $(c+1) \times (c+1)$-minors of the
matrix obtained from ${\rm Jac}(X)$ by adding the extra row $u$.
Let $J'$ denote the resulting ideal in $K[x_0,\ldots,x_n,u_0,\ldots,u_n]$.
In order to remove the singular locus of $X$
from the variety of $J'$, we replace $J'$ with the {\em saturation ideal}
\[
J \,\,:= \,\, \bigl( \,J' \,: \,  \langle \hbox{$\,c \times c$-minors of ${\rm Jac}(X)$}\, \rangle^\infty \bigr).
\]
See \cite[Exercise 8 in \S 4.4]{CLO} for the definition of saturation of ideals.

The ideal $J$ is bi-homogeneous in $x$ and $u$ respectively. Its
zero set in $\PP^n \times (\PP^n)^*$ is the
conormal variety ${\rm CN}(X) $. The ideal of the dual variety $X^*$
is finally obtained by eliminating the variables $x_0,\ldots,x_n$ from $J$:
\begin{equation}
\label{eq:idealofdual}
 \text{ideal of the dual variety $X^*$} \,\,\,  = \,\,\, J \,\cap \,K[u_0,u_1,\ldots,u_n]. 
 \end{equation}
As was remarked earlier, the expected dimension of $X^*$ is $n-1$, so 
the elimination ideal (\ref{eq:idealofdual}) is expected to be principal. We 
seek to compute its generator.
We shall see many examples of such dual hypersurfaces later on.

\begin{thm}{\rm (Biduality, \cite[Theorem 1.1]{GKZ})}
\label{thm:biduality} \\
Every irreducible projective variety $X \subset \PP^n$~satisfies
 \[  (X^*)^* \,\,= \,\,X .\]
\end{thm}

\begin{proof}[Proof idea]
The main step in proving this important theorem is that the 
conormal variety is self-dual, in the sense that
$\,{\rm CN}(X) = {\rm CN}(X^*)$. In this identity the roles of
$x \in \PP^n$ and $u \in (\PP^n)^*$ are swapped. It implies $ (X^*)^* = X$.
A proof for the self-duality of the conormal variety is found in  \cite[\S I.1.3]{GKZ}.
\end{proof}

\begin{example}
Suppose that  $X \subset \PP^n$ is a general smooth hypersurface of degree $d$.
 Then $X^*$ is a hypersurface of degree $d(d-1)^{n-1}$ in $(\PP^n)^*$. 
A concrete instance  for $ d=4$ and $n=2$ was seen in
Examples  \ref{ex:fourtwelve} and \ref{ex:fourtwelve2}. \qed
\end{example}

\begin{example}
\label{ex::conormal_SDP}
Let $X$ be the variety of symmetric $m\times m$ matrices of rank at most $r$. Then $X^*$ is the variety of symmetric $m\times m$ matrices of rank at most $m-r\,$ \cite[\S I.1.4]{GKZ}.
Here the conormal variety ${\rm CN}(X)$ consists of pairs of symmetric matrices $A$ and $B$
such that  $A\cdot B = 0$. This conormal variety will be important for our discussion of duality in semidefinite programming. \qed
\end{example}

An important class of examples, arising from toric geometry, is  featured in the
book by Gel'fand, Kapranov
and Zelevinsky \cite{GKZ}. A {\em projective toric variety} $X_A$ in 
$\PP^n$ is specified by an integer
matrix $A = (a_0,a_1,\ldots,a_n)$ of format $d \times (n{+}1)$ and rank $d$ whose
row space contains the vector $(1,1,\ldots,1)$. We define $X_A$ as the closure
in $\PP^n$ of the set $\,\bigl\{(t^{a_0}:t^{a_1}:\dots:t^{a_n}) \,|\, t\in (\C \backslash \{0\})^d \bigr\}$.

The dual variety $X_A^*$ is called the {\em $A$-discriminant}. It is usually a hypersurface,
in which case we identify the {\em $A$-discriminant} with the
irreducible polynomial $\Delta_A$ that vanishes on $X_A^*$.
The $A$-discriminant is indeed a discriminant in the sense that its vanishing
characterizes Laurent polynomials
\[
p(t)\,\,\, = \,\,\, \sum_{j=0}^n c_j \cdot t_1^{a_{1j}}t_2^{a_{2j}}\dots t_d^{a_{dj}}
\]
with the property that the hypersurface $\{p(t) = 0\}$ has a singular
point in $(\C \backslash \{0\})^d$. In other words, we can define (and compute) the $A$-discriminant~as
\[
\Delta_A  \,\,= \,\, \overline{\left\{ c \in (\PP^n)^* \,\, |\,\, \exists\, t \in (\C \backslash \{0\})^d \,\text{ with } 
\, p(t) = \frac{\partial p}{\partial t_1}=\dots= \frac{\partial p}{\partial t_d}=0\right\}}.
\]

\begin{example}
\label{ex:smalldiscriminant}
Let $d = 2 $, $n=4$, and fix the matrix
\[
A \quad = \quad \begin{pmatrix}
4 & 3 & 2 & 1 & 0\\
0 & 1 & 2 & 3 & 4 
\end{pmatrix}
\]
The associated toric variety is the rational normal curve 
\begin{align*} 
X_A & \,=\, \bigl\{ (t_1^4 : t_1^3 t_2 : t_1^2 t_2^2 : t_1 t_2^3 : t_2^4) \in \PP^4 
\,\,|\,\, (t_1:t_2) \in \PP^1 \bigr\} \\
 & \,=\, V(x_0 x_2 {-} x_1^2 , x_0 x_3 {-} x_1 x_2, x_0 x_4 {-} x_2^2,x_1 x_3 {-} x_2^2,
x_1 x_4 {-} x_2 x_3, x_2 x_4 {-} x_3^2).
\end{align*}
A hyperplane $\{ \sum_{j=0}^4 c_j x_j = 0 \} $ is tangent to $X_A$
if and only if the binary~form
\[ p(t_1,t_2) \quad = \quad  c_0 t_2^4 + c_1 t_1 t_2^3 + c_2 t_1^2 t_2^2 
+ c_3 t_1^3 t_2 + c_4 t_1^4
\]
has a linear factor of multiplicity $\geq 2$.
This is controlled by  the $A$-discriminant 
\begin{equation}
\label{discrofquartic} 
\Delta_A \,\, = \,\, \frac{1}{c_4} \cdot
{\rm det} \begin{pmatrix}
c_0 & c_1 & c_2 & c_3 & c_4 & 0 & 0   \\
0     & c_0 & c_1 & c_2 & c_3 & c_4 & 0   \\ 
0  & 0     & c_0 & c_1 & c_2 & c_3 & c_4   \\
c_1 & 2 c_2 & 3 c_3 & 4 c_4 & 0 & 0 & 0 \\
0 & c_1 & 2 c_2 & 3 c_3 & 4 c_4 & 0 & 0 \\
0 & 0 & c_1 & 2 c_2 & 3 c_3 & 4 c_4 & 0 \\
0 & 0 & 0 & c_1 & 2 c_2 & 3 c_3 & 4 c_4 
\end{pmatrix},
\end{equation}
given here in form of the determinant of a Sylvester matrix.
The sextic hypersurface $X_A^* = V(\Delta_A)$ is the
dual variety of the curve $X_A$.
  \qed
\end{example}

\section{The Optimal Value Function}
\label{sec::opt_value}

In this section we examine the optimization problem  \eqref{eq::nonlinearOPT}
under the hypotheses that the cost function $f(x)$ is linear and that there 
are no inequality constraints $g_i(x)$. The purposes of these restrictions is to
simplify the presentation and focus on the key ideas.
 Our analysis can be extended to the general problem
\eqref{eq::nonlinearOPT} and we discuss this briefly at the end of this section.

We consider the problem of
optimizing a linear cost function over a compact real
algebraic variety $X$ in $\R^n$:
\begin{equation}
\label{eq::varietyOPT}
\begin{aligned}
&c_0^* \,\,= \,\, \underset{x}{\text{minimize}} &&\langle c,x \rangle \\
& \text{subject to}
& & x \,\in\, X = \left\{v\in \R^n\,|\,h_1(v)=\dots=h_p(v) = 0\right\} .
\end{aligned}
\end{equation}
Here $h_1,h_2,\ldots,h_p$ are fixed polynomials in $n$ unknowns $x_1,\ldots,x_n$.
The expression
$\langle c, x \rangle  \,= \, c_1 x_1 + \cdots + c_n x_n$ is a linear form whose
coefficients $c_1,\ldots,c_n$ are unspecified parameters. Our aim is to compute
the {\em optimal value function} $c^*_0$. Thus, we regard the optimal value
$\,c_0^*\,$ as a function $\,\R^n \rightarrow \R\,$ of the parameters $c_1,\ldots,c_n$,
and we seek to determine this function.

The hypothesis that $X$ be compact has been included to ensure that the
optimal value function $c_0^*$ is well-defined on all of $\R^n$. Again, also
this hypothesis can be relaxed. We assume compactness here just for convenience.

Our problem is equivalent to that of describing the dual convex body
$P^\Delta$ of the convex hull $P = {\rm conv}(X)$, assuming that the latter
contains the origin in its interior. A small instance of this was seen in
(\ref{eq::simpleSDPlambda}). Since our convex hull $P$ is a semi-algebraic set,
Tarski's theorem on quantifier elimination in real algebraic geometry \cite{BPR, BCR} ensures
that the dual body $P^\Delta$ is also semialgebraic. This implies that the optimal value function
$c_0^*$ is an algebraic function, i.e., there exists a polynomial 
$\Phi(c_0,c_1,\ldots,c_n)$ in $n+1$ variables such that
\begin{equation}
\label{phiiszero}
 \Phi(c_0^*,c_1,\ldots,c_n) \quad = \quad 0 .
 \end{equation}
Our aim is to compute such a polynomial $\Phi$
of least possible degree. The input consists of
the polynomials $h_1,\ldots,h_p$ that cut out the variety $X$.
The degree of $\Phi$ in the unknown $c_0$ is called the
{\em algebraic degree} of the optimization problem \eqref{eq::nonlinearOPT}.
This number is an intrinsic algebraic complexity measure
for the problem of optimizing a linear function over $X$.
For instance, if $c_1,\ldots,c_n$ are rational numbers
then the algebraic degree indicates the degree of the
field extension $K$ over $\Q$ that contains
the coordinates of the optimal solution.

We illustrate our discussion by computing the optimal value function and 
its algebraic degree for the
 trigonometric space curve  
 featured in \cite[\S 1]{RS1}.

\begin{example}
Let $X$ be the curve in $\R^3$ with parametric representation
\[
(x_1,x_2,x_3) \quad = \quad \bigl( {\rm cos}(\theta), {\rm sin}(2\theta), {\rm cos}(3 \theta)\bigr). 
\]
In terms of equations, our curve can be written as $X = V(h_1,h_2)$, where
\[
h_1 \,\,= \,\,x_1^2-x_2^2-x_1 x_3 \quad \hbox{and} \quad h_2 \,\, = \,\,  x_3-4 x_1^3+3 x_1. 
\]
The optimal value function for maximizing $c_1 x_1 {+} c_2 x_2 {+} c_3 x_3$ over $X$
is given~by
\begin{small}
\[
\begin{matrix}
\!\!\!\! \!\!\! \!\!\! \Phi \quad =  \quad
(11664 c_3^4) \cdot c_0^6 \,\,+ \,\,
(864 c_1^3 c_3^3+1512 c_1^2 c_2^2 c_3^2
-19440 c_1^2 c_3^4 
\\ \qquad  \, +576 c_1 c_2^4 c_3
-1296 c_1 c_2^2  c_3^3+64 c_2^6-25272 c_2^2 c_3^4-34992 c_3^6)\cdot c_0^4
\\  + \,
(16 c_1^6 c_3^2+8 c_1^5 c_2^2 c_3-1152 c_1^5 c_3^3-1920 c_1^4 c_2^2 c_3^2+8208  c_1^4 c_3^4 -724 c_1^3 c_2^4 c_3 
 +144 c_1^3 c_2^2 c_3^3 \\  +c_1^4 c_2^4
-17280 c_1^3 c_3^5-80 c_1^2 c_2^6- 2802 c_1^2 c_2^4 c_3^2-3456 c_1^2 c_2^2 c_3^4+3888 c_1^2 c_3^6
-1120 c_1 c_2^6 c_3
\\ 
+540 c_1 c_2^4 c_3^3+55080 c_1 c_2^2 c_3^5-128 c_2^8-208 c_2^6 c_3^2{+}15417 c_2^4 c_3^4
{+}15552 c_2^2 c_3^6{+}34992 c_3^8) \cdot c_0^2
\\  + \,
(-16 c_1^8 c_3^2-8 c_1^7 c_2^2 c_3+256 c_1^7 c_3^3-c_1^6 c_2^4+328 c_1^6 c_2^2 c_3^2-1600  c_1^6 c_3^4+114 c_1^5 c_2^4 c_3
\\ 
-2856 c_1^5 c_2^2 c_3^3+4608 c_1^5 c_3^5+12 c_1^4 c_2^6-1959 c_1^4 c_2^4 c_3^2+9192 c_1^4 c_2^2 c_3^4-4320 c_1^4 c_3^6
\\ 
-528 c_1^3 c_2^6 c_3+7644 
c_1^3 c_2^4 c_3^3-7704 c_1^3 c_2^2 c_3^5-6912 c_1^3 c_3^7-48 c_1^2 c_2^8+3592 c_1^2 c_2^6  c_3^2 
\\ 
-4863 c_1^2 c_2^4 c_3^4-13608 c_1^2 c_2^2 c_3^6+15552 c_1^2 c_3^8+800 c_1 c_2^8 c_3-400 c_1 c_2^6 c_3^3-10350 c_1 c_2^4 c_3^5
 \\  +16200 c_1 c_2^2 c_3^7
+64 c_2^10+80 c_2^8 c_3^2-1460 c_2^6 c_3^4+135 c_2^4 c_3^6+9720 c_2^2 c_3^8-11664 c_3^{10}).
\end{matrix}
\]
\end{small}
The optimal value function $c_0^*$ is the algebraic
function of $c_1,c_2,c_3$ obtained by solving $\Phi = 0$ for 
the unknown $c_0$. Since $c_0$ has degree $6$ in $\Phi$, we see that
the algebraic degree of this optimization problem is $6$.
Note that we can write $c_0^*$ in terms of radicals in $c_1,c_2,c_3$ because
there are no odd powers of $c_0$ in $\Phi$, which ensures that the Galois group 
of $c_0^*$ over $\Q(c_1,c_2,c_3)$ is solvable. \qed
\end{example}

We now come to the main result in this section.
It will explain what the polynomial $\Phi$ means
and how it was computed in the previous example. For the sake of simplicity, we
shall first assume that the given variety $X$ is smooth, 
i.e. $X=X_{\rm reg}$, where the set $X_{\rm reg}$ denotes all 
regular points on $X$. 

\begin{thm} \label{thm:dualoptvalue}
Let $X^* \subset (\PP^n)^*$  be the dual variety to the projective closure of $X$. 
If $X$ is irreducible, smooth and compact in $\R^n$ then
$X^*$ is an irreducible hypersurface, and its defining polynomial
equals $\Phi(-c_0,c_1,\ldots,c_n)$ where $\Phi$ represents
the optimal value function as in \eqref{phiiszero} of the optimization problem
 \eqref{eq::varietyOPT}. In particular, the 
algebraic degree of \eqref{eq::varietyOPT} is the degree 
in $c_0$ of the irreducible polynomial that vanishes on the dual hypersurface $X^*$.
\end{thm}

Here the change of sign in the coordinate $c_0$ is needed because
the equation $\,c_0 = c_1x_1 +\cdots + c_n x_n\,$ for the objective
function value in $\R^n$ becomes the homogenized  
equation  $\,(-c_0) x_0 + c_1 x_1 + \cdots + c_n x_n = 0 \,$ 
when we pass to
  $\PP^n$.

\begin{proof}
Since $X$ is compact, for every cost vector $c$ there exists an optimal solution $x^*$.
Our assumption that $X$ is smooth ensures that $x^*$ is a regular point of $X$, and 
$c$ lies in the span of the gradient vectors $\nabla_x h_i\big|_{x^*}$ for $i=1,\ldots,p$.
In other words, the KTT conditions are necessary at the point $x^*$:
\begin{align}\label{eqn::KKT_equality}
c &\,\,=\, \sum_{i=1}^{p} \lambda_i^* \cdot \nabla_x  h_i\big|_{x^*},  \nonumber \\
h_i(x^*) &\,=\,\,\, 0 \qquad \hbox{for}\,\,\, \,i=1,2,\ldots,p .\nonumber
\end{align}
The scalars $\lambda_1^*, \ldots, \lambda_p^*$ express $c$ as a vector in the orthogonal
complement of the tangent space of $X$ at $x^*$.
In other words, the affine hyperplane
$\,\{x \in \R^n \,:\, \langle c, x \rangle = c_0^*\}$
contains the tangent space of $X$ at $x^*$.
This means that the pair $\,\bigl(\,x^*\,,\, (-c_0^*:c_1:\cdots :c_n) \,\bigr)\,$
lies in the conormal variety ${\rm CN}(X) \subset
\PP^n \times (\PP^n)^*$ of the projective closure of $X$.
By projection onto the second factor, we see that
$\,(-c_0^*:c_1:\cdots :c_n)\,$ lies in the dual variety~$X^*$.

Our argument shows that the boundary of 
the dual body $P^\Delta$ is a subset of $X^*$.
Since that boundary is a semialgebraic set of dimension $n-1$,
we conclude that $X^*$ is a hypersurface. If we write its
defining equation as $\Phi(-c_0,c_1,\ldots,c_n) = 0$, then the polynomial $\Phi$
satisfies \eqref{phiiszero}, and the
 statement about the algebraic degree follows as well.
\end{proof}

The KKT condition for the optimization problem
 \eqref{eq::varietyOPT} involves three sets of variables,
 two of which are dual variables,
to be carefully distinguished:
 \begin{enumerate}
\item Primal variables $x_1,\ldots,x_n$ to describe the set $X$ of feasible solutions.
\item (Lagrange) dual variables $\lambda_1,\ldots,\lambda_p$ to parametrize
the linear space of all hyperplanes that are tangent to $X$ at a fixed point $x^*$.
\item (Projective) dual variables $c_0,c_1,\ldots,c_n$ 
for the space of all hyperplanes. These
are coordinates for the dual variety $X^*$ and the dual body~$P^\Delta$.
\end{enumerate}
We can compute the equation $\Phi$ that defines the dual hypersurface $X^*$
by eliminating the first two groups of variables 
$x = (x_1,\ldots,x_n)$ and $\lambda=(\lambda_1,\dots,\lambda_p)$ from the 
following system of polynomial equations:
\[
c_0 \,= \,\langle c, x \rangle  \,\,\, \hbox{and} \,\,\, h_1(x) = \cdots = h_p(x) = 0 \,\,\, \hbox{and} \,\,\, c \,= \,\lambda_1 \nabla_x h_1 + \cdots + \lambda_p \nabla_x h_p. 
\]

\begin{example}{\rm (Example~\ref{ex:fourtwelve} cont.)} \label{ex:quartic1}
We consider 
\eqref{eq::varietyOPT} with $n=2$, $p = 1$ and $h_1 = x_1^4+x_2^4-1$.
The KKT equations for maximizing the function
\begin{equation}
\label{exeq1}
c_ 0 \,\,= \,\, c_1 x_1 + c_2 x_2  
\end{equation}
over the ``TV screen'' curve $X = V(h_1)$ are
\begin{equation}
\label{exeq2}
c_1 \,=\, \lambda_1 \cdot 4 x_1^3\, ,\quad 
c_2 \,=\, \lambda_1 \cdot 4 x_2^3\, ,\quad
x_1^4 + x_2^4 = 1.
\end{equation}
We eliminate the three unknowns $x_1,x_2,\lambda_1$
from the system of four polynomial equations in (\ref{exeq1}) and (\ref{exeq2}).
The result is the polynomial $\Phi(-c_0,c_1,c_2)$ of degree $12$
which expresses the optimal value $c_0^*$  as an algebraic
function of $c_1$ and $c_2$. We note that $\Phi(1,c_1,c_2)$
is precisely the polynomial in  (\ref{eq:curvedeg12}). \qed
\end{example}

It is natural to ask what happens with Theorem \ref{thm:dualoptvalue}
when $X$ fails to be smooth or compact,
or if there are additional inequality constraints.
Let us first consider the case when $X$ is no longer smooth, 
but still compact. Now, $X_{\rm reg}$ is a proper (open, dense) subset of $X$.
The optimal value function $c_0^*$ for the problem
\eqref{eq::varietyOPT} is still perfectly well-defined on all of
$\R^n$, and  it is still an algebraic function of $c_1,\ldots,c_n$.
However, the polynomial $\Phi$ that represents $c_0^*$
may now have more factors than just the equation of
the dual variety $X^*$.

\begin{example} \label{ex:quartic2}
Let $n=2$ and $p=1$ as in Example \ref{ex:quartic1}, but
now we consider a singular quartic.
The {\em bicuspid curve}, shown in Figure~\ref{fig::bicuspid}, 
has the equation
\[
h_1 \,\,= \,\, (x_1^2 - 1)(x_1 - 1)^2 + (x_2^2 - 1)^2.
\]
\begin{figure}
\centering
\includegraphics[width=0.5\textwidth]{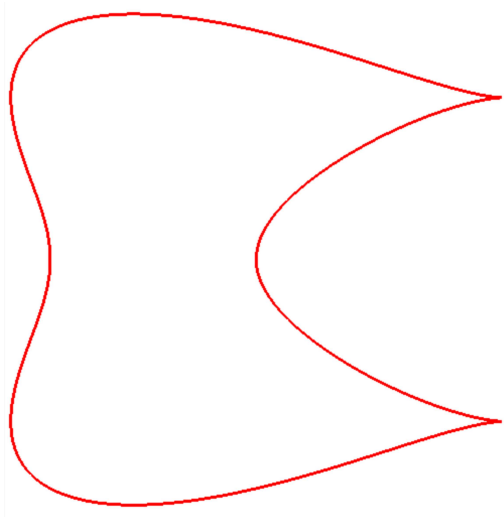}
\caption{The bicuspid curve in Example \ref{ex:quartic2}.\label{fig::bicuspid}}
\end{figure}
The algebraic degree of optimizing a linear function  
$c_1 x_1 + c_2 x_2$ over $X = V(h_1)$ equals $8$.
The optimal value function $c_0^* = c_0^*(c_1,c_2)$ is represented by
\[ \begin{matrix}
\Phi &=& 
\bigl(c_ 0-c_ 1+c_ 2 \bigr) \cdot
\bigl(c_ 0-c_ 1-c_ 2 \bigr) \cdot
\bigl(
16 c_ 0^6 - 48 (c_ 1^2+ c_ 2^2) c_ 0^4+
\\ & & (24 c_ 1^2 c_ 2^2+21 c_ 2^4+64 c_ 1^4) c_ 0^2
  +(54 c_ 1 c_ 2^4{+}32 c_ 1^5) c_ 0+8 c_ 1^4 c_ 2^2 {-} 3 c_ 1^2 c_ 2^4 {+} 11 c_ 2^6
\bigr).
\end{matrix}
\]
The first two linear factors correspond to the singular points of the bicuspid curve
$X$ , and the larger factor of degree six represents the
dual curve $X^*$. \qed
\end{example}

This example shows that, when $X$ has singularities,
it does not suffice to just dualize the variety $X$ but we must also
dualize the singular locus of $X$. This process is recursive, and we must also
consider the singular locus of the singular locus etc.  We believe that,
in order to characterize the value function $\Phi$, it always
suffices to dualize all irreducible varieties occurring in a  {\em Whitney stratification}
of $X$ but this has not been worked out yet. In our view, this topic 
requires more research, both on the theoretical side
and on the computational side. The following
result is valid for any variety~$X$ in $\R^n$.

\begin{cor}
If the dual variety of $X$ is a hypersurface
then its defining polynomial contributes a factor
to the value function of the problem (\ref{eq::varietyOPT}).
\end{cor}

This result can be extended to an arbitrary optimization problem of the form \eqref{eq::nonlinearOPT}.
We obtain a similar characterization of the optimal value $c_0^*$ as a semi-algebraic function of
$c_1,c_2,\ldots,c_n$  by eliminating all primal variables $x_1,\ldots,x_n$ 
and all dual (optimization) variables $x, \lambda, \mu$ from the KKT equations.
Again, the optimal value function is represented  by a unique square-free polynomial
$\Phi(c_0,c_1,\ldots,c_n)$, and each factor of this polynomial is the dual hypersurface
$Y^*$ of some variety $Y$ that is obtained from $X$ by setting $g_i(x) = 0$
for some of the inequality constraints, by recursively passing to singular loci.
In Section~\ref{sec::sdp} we shall explore this for semidefinite programming.

We close this section with a simple example involving
$A$-discriminants.

\begin{example}
\label{ex:calculus}
Consider the
calculus exercise of minimizing a  polynomial
\[ q(t) \quad = \quad c_1 t+ c_2t^2+c_3t^3+c_4 t^4\]
of degree four over the real line $\R$.
Equivalently, we wish to minimize
\[
\,c_0 \,= \,c_1 x_1 + c_2 x_2 + c_3 x_3 + c_4 x_4 
\]
over the {\em rational normal curve} $X_A \cap \{x_0 = 1\} \,=\, 
V(x_1^2-x_2,x_1^3-x_3,x_1^4-x_4)$,
seen in Example \ref{ex:smalldiscriminant}. The optimal value
function $c_0^*$ is given by the equation
$\Delta_A(-c_0,c_1,c_2,c_3,c_4) = 0$,
where $\Delta_A$ is the discriminant in (\ref{discrofquartic}).
Hence the algebraic degree of this optimization problem is equal to 
three. \qed
\end{example}

\section{An Algebraic View of Convex Hulls}
\label{sec::convex_hull}
The problem of optimizing arbitrary linear functions over a given subset of $\R^n$, discussed in the
previous section, leads naturally to the geometric question of how to represent the
convex hull of that subset. In this section we explore this question
from an algebraic perspective. To be precise, we shall study the algebraic boundary $\partial_a P$ of the convex hull $\,P = \conv(X)\,$ of 
a compact real algebraic variety $X$ in $\R^n$.
 Biduality of projective varieties  (Theorem \ref{thm:biduality})
  will play an important
role in understanding the structure of $\partial_a P$.
The results to be presented are drawn from  \cite{RS1, RS2}.
In Section~\ref{sec::spec_shadow} we shall discuss 
the alternative representation of
$P$ as a spectrahedral shadow.

We begin with the seemingly easy example of a plane quartic curve.

\begin{example}
\label{ex:trottcurve}
We consider the following smooth compact plane curve
\begin{equation}
\label{eq:trottcurve}
  X \,\, = \,\, \bigl\{
(x,y) \in \R^2 \,\,|\,\, 144 x^4+144 y^4-225 (x^2+y^2) + 350 x^2 y^2+81\,= \,0 \,\bigr\}. 
\end{equation}

\begin{figure}
\vskip -0.4cm \qquad \qquad \qquad
\includegraphics[width=0.6\textwidth]{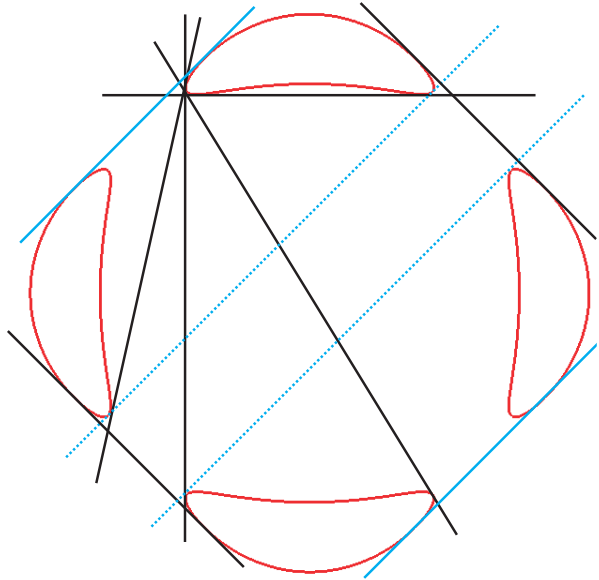}
\vskip -0.4cm
\caption{A quartic curve in the plane can have up to $28$ real bitangents.\label{fig:quartic}}
\vskip -0.1cm
\end{figure}

This curve is known as the {\em Trott curve}. It was first constructed
by Michael Trott in \cite{Tro}, and is illustrated above in Figure \ref{fig:quartic}.
A classical result of algebraic geometry states that a general quartic curve
in the complex projective plane $\PP^2$ has $28$ bitangent lines, and the
Trott curve $X$ is an instance where all $28$ lines are real and have a 
coordinatization in terms of radicals over $\Q$.
Four of the $28$ bitangents form edges of $ {\rm conv}(X)$.
These special bitangents~are
\[
\{(x,y) \in \R^2\,|\, \pm \,x \,\pm\, y  \,=\, \gamma\},
\quad \hbox{where} \,\,\,\gamma \,=\, \frac{ \sqrt{48050+434 \sqrt{9889}} }{248} \,= \,1.2177... 
\]
The boundary of ${\rm conv}(X)$ alternates between
these four edges and pieces of the curve $X$.
The eight transition points have the floating point coordinates
\[       (\,\pm \,0.37655..., \pm \,0.84122...) \,\,,\,\,\,
         (\,\pm \,0.84122..., \pm \,0.37655...). \]
These coordinates lie in the field $\Q(\gamma)$
and we invite the reader to write them in the form
$q_1 + q_2 \gamma$ where $q_i \in \Q$.
The $\Q$-Zariski closure of the $4$ edge lines of ${\rm conv}(X)$ is
a curve $Y$ of degree $8$. Its equation
has two irreducible factors:
\[
\begin{matrix}
 (992 x^4{-}3968 x^3 y{+}5952 x^2 y^2{-}3968 x y^3{+}992 y^4{-}1550 x^2{+}3100 x y {-}1550 y^2{+}117) \\
 (992 x^4{+}3968 x^3 y{+}5952 x^2 y^2{+}3968 x y^3{+}992 y^4{-}1550 x^2{-}3100 x y{-}1550 y^2{+}117)
\end{matrix}
\]
Each reduces over $\R$ to four parallel lines, two of which contribute to the 
boundary.
The point of this example is to stress the role of the (arithmetic of) bitangents
in any exact description of the convex hull of a plane curve. \qed
\end{example}

We now present a general formula for the algebraic boundary of the convex hull
of a compact variety $X$ in $\R^n$. The key observation is that the algebraic 
boundary of $P = {\rm conv}(X)$ will consist of different types of components, 
resulting from planes that are 
simultaneously tangent at $k$ different points of $X$, for various values of 
the integer $k$. For the Trott curve $X$ in Example \ref{ex:trottcurve}, the 
relevant integers were $k=1$ and $k=2$, and we demonstrated that the algebraic
boundary of its convex hull $P$ is a reducible curve of degree $12$:
\begin{equation}
\label{eq:trottboundary}
 \partial_a (P) \,\,\,= \,\,\, X \,\cup\, Y .
 \end{equation}
In the following definitions we regard $X$ as
a complex projective variety in~$\PP^n$. 

Let $X^{[k]}$ be the variety in the dual projective space $(\PP^n)^*$
which is the closure of the set of all hyperplanes that are tangent to $X$ 
at $k$ regular points which span a $(k-1)$-plane in $\PP^n$. This definition 
makes sense for $k=1,2,\ldots,n$. Note that  $X^{[1]}$ coincides 
with the dual variety $X^*$, and $X^{[2]}$ parametrizes all hyperplanes
that are tangent to $X$ at two distinct points. Typically,
$X^{[2]}$ is an irreducible component of 
the singular locus of $X^* = X^{[1]}$.
We have the following nested chain of projective varieties 
in the dual space:
\[ X^{[n]} \,\subseteq \, X^{[n-1]}\, \subseteq \,
\cdots \, \subseteq \, X^{[2]} \,\subseteq \,X^{[1]}
\,\,\subseteq \,\, (\PP^n)^*. \]
We now dualize each of the varieties in this chain. The resulting
varieties $(X^{[k]})^*$ live in the primal projective space $\PP^n$.
For $k=1$ we return to our original variety, i.e.,~we have $\, (X^{[1]})^* \, = \, X\,$
by biduality (Theorem \ref{thm:biduality}). In the following result
we assume that $X$ is smooth as a complex variety in $\PP^n$,
and we require one technical hypothesis concerning tangency of hyperplanes.

\begin{thm}{\rm \cite[Theorem 1.1]{RS2}}
\label{boundarytheorem}
Let $X$ be a smooth and compact real algebraic variety that affinely spans $\R^n$,
and such that only finitely many hyperplanes are tangent to $X$ at infinitely
many points.
The algebraic boundary $\partial_a P$ of its convex hull, P = conv(X), can be
 computed by biduality as follows:
\begin{equation}
\label{eq:boundaryformula}
\partial_aP \,\,\subseteq \,\, \bigcup_{k =1}^n (X^{[k]})^*.
\end{equation}
\end{thm}

Since $\partial_a P$ is pure of codimension one, 
in the union we only need indices $k$ having
property that $ (X^{[k]})^* $ is a hypersurface in $\PP^n$.
As argued in \cite{RS2}, this leads to the following
lower bound on the relevant values to be considered:
\begin{equation}
\label{eq:lowerbound}
 k \,\,\geq \,\, \big\lceil \frac{n}{{\rm dim}(X)+1} \big\rceil . 
 \end{equation}
The formula  (\ref{eq:boundaryformula}) computes
the algebraic boundary $\partial_a P$ in the following sense.
For each relevant $k$ we check whether $(X^{[k]})^*$ is
a hypersurface, and, if yes, we determine its irreducible components
(over the field $K$ of interest).
For each component we then check, usually by means of
numerical computations, whether it meets the boundary $\partial P$
in a regular point. The irreducible hypersurfaces which survive this
test are precisely the components of $\partial_a X$.

\begin{example} 
\label{ex:convhullcurve}
When $X$ is a plane curve in $\R^2$ then (\ref{eq:boundaryformula}) says that
\begin{equation}
\label{eq:apXplane}
 \partial_a P \,\, \subseteq \,\, X \cup (X^{[2]})^* .
 \end{equation}
Here $X^{[2]}$ is the set of points in $(\PP^2)^*$ that are dual
to the bitangent lines of $X$, and $(X^{[2]})^*$ is the union 
of those lines in $\PP^2$. If we work over $K = \Q$ 
and the curve $X$ is general enough then we expect
equality to hold in  (\ref{eq:apXplane}). For special curves
the inclusion can be strict. This happens for the Trott curve (\ref{eq:trottcurve})
since $Y$ is a proper subset of $(X^{[2]})^*$. Namely, 
$Y$ consists of two of the six $\Q$-components of $(X^{[2]})^*$.
However, a small perturbation of the coefficients in
(\ref{eq:trottcurve}) leads to a curve $X$ with
equality in (\ref{eq:apXplane}), as the relevant
Galois group acts transitively on the $28$ points in $X^{[2]}$ for general
quartics~$X$. Now, the algebraic boundary over $\Q$ is
a reducible curve of degree $32  = 28+4$. \qed
\end{example}

If we are given the variety $X $ in terms of equations or  in parametric form, then we
can compute equations for $X^{[k]}$ by an elimination process similar b of
the dual variety $X^*$.  However, expressing the tangency condition at $k$ different points requires a larger number of additional variables (which need to be eliminated afterwards) and thus the computations are quite involved. The subsequent step of dualizing $X^{[k]}$ to get
the right hand side of (\ref{eq:boundaryformula}) is even more forbidding. The resulting
hypersurfaces  $(X^{[k]})^*$ tend to have  high degree and their
defining polynomials are very large when $n \geq 3$.

The article \cite{RS1} offers a detailed study of the case when
$X$ is a space curve in $\R^3$. Here the lower bound (\ref{eq:lowerbound})
tells us that $\,\partial_a X \, \subseteq \, (X^{[2]})^* \cup (X^{[3]})^*$.
The surface $(X^{[2]})^*$ is the {\em edge surface} of the curve $X$,
and $(X^{[3]})^*$ is the union of all {\em tritangent planes} of $X$.
The following example illustrates these objects.

\begin{figure}[htb]
\centering\includegraphics[width=0.6\textwidth]{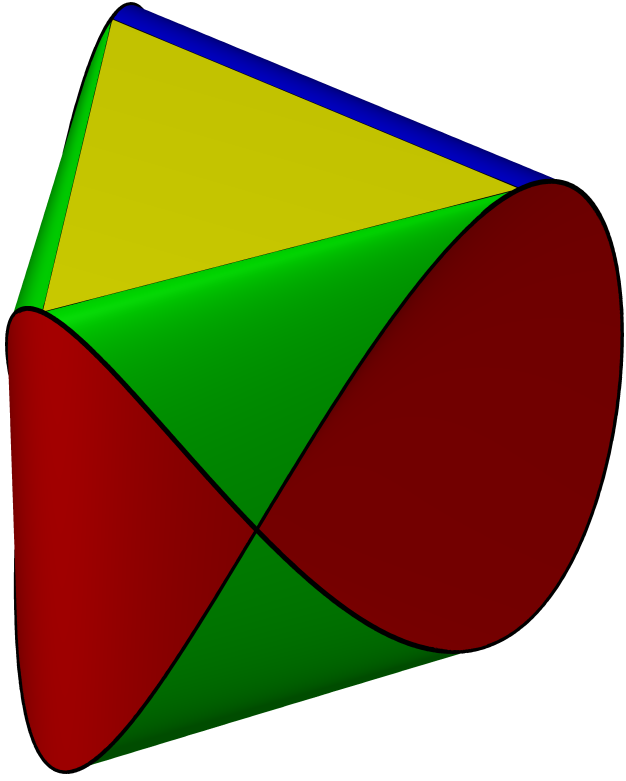}
\caption{The convex hull of the curve $\,( {\rm cos}(\theta) ,  {\rm cos}(2 \theta),  {\rm sin}(3 \theta)) \, $ in $\,\R^3$.\label{fig::trigonometric}}
\end{figure}

\begin{example}
\label{ex::trigonometric}
We consider the trigonometric curve $X$ in $\R^3$ which has the
parametrization $x = {\rm cos}(\theta) $,  $y= {\rm cos}(2 \theta)$,  $z = {\rm sin}(3 \theta)$.
This is an algebraic curve of degree six. Its implicit representation
equals $X=V(h_1,h_2)$ where
\[ h_1 \,=\, 2 x^2-y-1 \,\,\, \hbox{and} \,\,\,
h_2 \,= \, 4 y^3+2 z^2-3y-1. \]
The edge surface $(X^{[2]})^*$ has three irreducible components. Two of the
components are the quadric $V(h_1)$ and the cubic $V(h_2)$. The
third and most interesting component of $(X^{[2]})^*$ is the surface of degree $16$
with equation~$\,h_3 = $
{\allowdisplaybreaks \begin{footnotesize}
\begin{align*}
&-419904x^{14}y^2+664848x^{12}y^4-419904x^{10}y^6+132192x^8y^8-20736x^6y^{10}
+1296x^4y^{12}  \\ &
-46656x^{14}z^2+373248x^{12}y^2z^2-69984x^{10}y^4z^2 {-} 22464x^8y^6z^2
{+} 4320x^6y^8z^2 {+} 31104x^{12}z^4
\\ & +5184x^{10}y^2z^4+4752x^8y^4z^4+1728x^{10}z^6
+699840x^{14}y-46656x^{12}y^3-902016x^{10}y^5 \\ &
{+} 694656x^8y^7 {-} 209088x^6y^9
{-} 1150848x^{10}y^3z^2{+}279936x^8y^5z^2{+}17280x^6y^7z^2 {-} 4032x^4y^9z^2
\\ & -98496x^{10}yz^4
+27072x^4y^{11}-1152x^2y^{13}-419904x^{12}yz^2-25920x^8y^3z^4-4608x^6y^5z^4\\
&-1728x^8yz^6-291600x^{14}-169128x^{12}y^2-256608x^{10}y^4+956880x^8y^6
-618192x^6y^8
\\ & +148824x^4y^{10}-13120x^2y^{12}+256y^{14}+392688x^{12}z^2 
+671976x^{10}y^2z^2
+1454976x^8y^4z^2 \\ & -292608x^6y^6z^2-4272x^4y^8z^2+1016x^2y^{10}z^2
{-}116208x^{10}z^4{+}135432x^8y^2z^4 {+}18144x^6y^4z^4 \\ & +1264x^4y^6z^4-5616x^8z^6
+504x^6y^2z^6-1108080x^{12}y+925344x^{10}y^3+215136x^8y^5
\\ & -672192x^6y^7
+331920x^4y^9-54240x^2y^{11}+2304y^{13}{+}273456x^{10}yz^2{+}282528x^8y^3z^2\\
&-1185408x^6y^5z^2+149376x^4y^7z^2-368x^2y^9z^2-32y^{11}z^2 {+} 273456x^8yz^4
{-} 67104x^6y^3z^4 \\ &
-4704x^4y^5z^4-64x^2y^7z^4+4752x^6yz^6-32x^4y^3z^6+747225x^{12}
+636660x^{10}y^2 \\ & 
-908010x^8y^4-65340x^6y^6+291465x^4y^8-101712x^2y^{10}+8256y^{12}-818100x^{10}z^2
\\ & -1405836x^8y^2z^2-905634x^6y^4z^2+583824x^4y^6z^2
-39318x^2y^8z^2+368y^{10}z^2 {+}193806x^8z^4
\\ & -282996x^6y^2z^4+15450x^4y^4z^4
+716x^2y^6z^4+y^8z^4+6876x^6z^6-1140x^4y^2z^6+2x^2y^4z^6
\\ & + x^4z^8+507384x^{10}y
-809568x^8y^3+569592x^6y^5-27216x^4y^7-71648x^2y^9+13952y^{11} \\ & +555768x^8yz^2
+869040x^6y^3z^2+688512x^4y^5z^2-154128x^2y^7z^2 {+} 4416y^9z^2 {-} 343224x^6yz^4\\
&+127360x^4y^3z^4-1656x^2y^5z^4-64y^7z^4-4536x^4yz^6{+}48x^2y^3z^6{-}775170x^{10}
{-}191808x^8y^2 \\ & +599022x^6y^4
 -245700x^4y^6+31608x^2y^8+7872y^{10}+765072x^8z^2
+589788x^6y^2z^2 \\ & -66066x^4y^4z^2-234252x^2y^6z^2+16632y^8z^2-173196x^6z^4
+248928x^4y^2z^4-26158x^2y^4z^4 \\ & -32y^6z^4-3904x^4z^6+804x^2y^2z^6+2y^4z^6
-2x^2z^8+5832x^8y+98280x^6y^3-219456x^4y^5 \\ & +72072x^2y^7   -8064y^9
-724032x^6yz^2
-515760x^4y^3z^2-99672x^2y^5z^2+29976y^7z^2
\\ & +225048x^4yz^4
-76216x^2y^3z^4+1912y^5z^4
+1696x^2yz^6-32y^3z^6+411345x^8-66096x^6y^2 \\ & {-}62532x^4y^4 {+} 29388x^2y^6
 {-} 11856y^8 {-} 365346x^6z^2 {+}19812x^4y^2z^2 {+} 104922x^2y^4z^2 {+} 24636y^6z^2
\\ & +85090x^4z^4
{-}104580x^2y^2z^4
{+} 8282y^4z^4{+}1014x^2z^6{-}144y^2z^6+z^8{-}39744x^6y {+} 61992x^4y^3
\\ & +2304x^2y^5+576y^7
+305328x^4yz^2+86640x^2y^3z^2+960y^5z^2-73480x^2yz^4+16024y^3z^4 \\ & -200yz^6
-114966x^6+24120x^4y^2-5958x^2y^4+6192y^6+85494x^4z^2-39696x^2y^2z^2\\
&-11970y^4z^2-21610x^2z^4+16780y^2z^4-94z^6-3672x^4y-11024x^2y^3+272y^5\\
&-46904x^2yz^2-4632y^3z^2+9368yz^4+15246x^4-84x^2y^2-1908y^4-6892x^2z^2\\
&+2204y^2z^2+2215z^4+3216x^2y+168y^3+904yz^2-664x^2+292y^2-282z^2-96y+9.
\end{align*}
\end{footnotesize}}
The boundary of $P = {\rm conv}(X)$ contains patches from all three surfaces
$V(h_1)$, $V(h_2)$ and $V(h_3)$. There are also two triangles, with
 vertices  at $(\sqrt{3}/2,1/2,\pm 1)$,
$(\sqrt{3}/2,1/2,\pm1)$ and $(0,-1,\pm 1)$. They span two
of the  tritangent planes of $X$, namely, $z=1$ and $z=-1$.
The union of all tritangent planes equals $(X^{[3]})^*$.
Only one triangle is visible in Figure~\ref{fig::trigonometric}. It is colored in yellow.
The curved blue patch adjacent to one of the edges of the triangle is given by the cubic $h_3$,
while the other two edges of the triangle lie in the degree $16$ surface $V(h_3)$.
The curve $X$ has two singular points at $(x,y,z) = (\pm 1/2, -1/2, 0)$. Around these
two singular points, the boundary is given by four alternating patches
from the quadric $V(h_1)$ highlighted in red and the 
degree $16$ surface $V(h_3)$ in green. We conclude that the
edge surface  $\,(X^{[2]})^* = V(h_1 h_2 h_3)\,$ is reducible
of degree $21 = 2 + 3 + 16$, and the algebraic boundary 
$\partial_a(P) $ is a reducible surface of degree $\,23= 2 + 21 $. \qed
\end{example}

In our next example we examine the convex hull of
space curves of degree four that are obtained as the
intersection of two quadratic surfaces in $\R^3$.

\begin{example}
\label{ex:ellipticcurve}
Let $X = V(h_1,h_2)$ be the intersection of two quadratic surfaces in $3$-space.
We assume that $X$ has no singularities in $\PP^3$. Then $X$ is a curve of genus one.
According to recent work of Scheiderer \cite{Sch}, the convex body 
$P = {\rm conv}(X)$ can be represented exactly using
Lasserre relaxations, a topic we shall return to when
discussing spectrahedral shadows in Section~\ref{sec::spec_shadow}. 
If we are willing to work over $\R$ then $P$ is in fact
a spectrahedron, as shown in \cite[Example 2.3]{RS1}.
We here derive that representation for a concrete example.

Lazard {\it et al.} \cite[\S 8.2]{LPP} examine the curve $X$ cut out by the two quadrics
\[ h_1 = x^2 + y^2 + z^2 - 1
\quad {\rm and} \quad
h_2 = 19 x^2 + 22 y^2 + 21 z^2 - 20. 
\]
Figure~\ref{fig::lazard_variety} shows the two components of
$X$ on the unit sphere $V(h_1)$.

\begin{figure} [hb]
\centering
\includegraphics[width=0.5\textwidth]{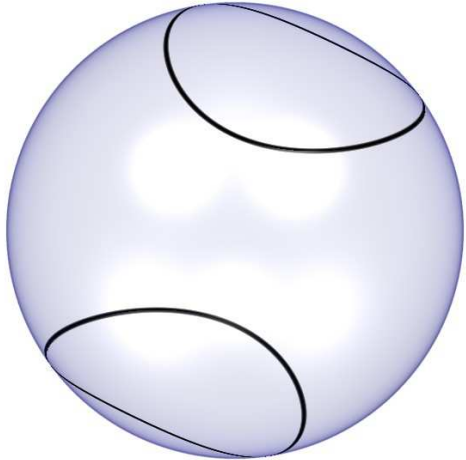}
\vskip -0.6cm
\caption{The curve on the unit sphere discussed in Example~\ref{ex:ellipticcurve} and
\ref{ex:ellipticcurve2}.
\label{fig::lazard_variety}
}
\end{figure}

The dual variety  $X^*$ is a surface of degree $8$ in $(\PP^3)^*$.
The singular locus of $X^*$ contains the curve $X^{[2]}$ which 
is the union of four quadratic curves. The duals of these four plane curves are the singular
quadratic surfaces defined~by
\[
h_3 \, = \, x^2-2 y^2-z^2 , \,\,
h_4 \, = \, 2 x^2-y^2-1 ,\,\,
h_5 \, = \, 3 y^2+2 z^2-1 ,\,\,
h_6 \, = \, 3 x^2+z^2-2.
\]
The edge surface of $X$ is the union of these four quadrics:
\[
(X^{[2]})^* \,\, = \,\, V(h_3) \,\cup \, V(h_4) \,\cup \,V(h_5) \,\cup \, V(h_6) . 
\]
The algebraic boundary of $P$ consists of the last two among these quadrics:
\[
\partial_a P \,\, = \,\, V(h_5) \,\cup \, V(h_6).
\]
These two quadrics are convex.
From this we derive a representation of $P$ as a spectrahedron by applying Schur
complements to the quadrics $h_5$ and $h_6$:
\[
P \quad = \quad \biggl\{
(x,y,z) \in \R^3 \,\,\bigg|\,\,
\begin{pmatrix}  1 {+} \sqrt{3} y & \sqrt{2} z & 0 & 0 \\
                                   \sqrt{2} z & 1{-} \sqrt{3} y & 0 & 0 \\
                                    0 & 0 & \sqrt{2} {+} x & \sqrt{3} x \\
                                     0 & 0 & \sqrt{3} x & \sqrt{2} {-} x 
\end{pmatrix} \,\succeq \, 0 \,\biggr\}.
\]
\end{example}

\section{Spectrahedra and Semidefinite Programming}
\label{sec::sdp}

Spectrahedra and semidefinite programming (SDP) have already surfaced
a few times in our discussion. In this section we take a systematic look 
at these topics and their dualities.
We write $\SSS^n$ for the space of real symmetric $n {\times} n$-matices
and $\SSS_+^n$ for the cone of positive semidefinite matrices in $\SSS^n
\simeq \R^{\binom{n+1}{2}}$.
This cone is self-dual with respect to the inner product
$\langle U,V \rangle = {\rm trace}(U \cdot V)$.

A {\em spectrahedron} is the intersection of the  cone $\SSS_+^n$ 
with an affine subspace
\[
\mathcal{K} \,\,= \,\,C \,+\,\underbrace{\Span(A_1,A_2,\ldots,A_m)}_{\mathcal{W}}. 
\]
Here $\mathcal{W}$  is a linear subspace of dimension $m$
in $\SSS^n$, and the spectrahedron is 
\begin{equation}
\label{spectraP}
 P \,\,\, = \, \,\,\,  \bigl\{\,x\in \R^m \,\,|\,\,  C - \sum_{i=1}^m x_i A_i \succeq 0 \bigr\} 
\,\,\, \simeq \,\,\, \mathcal{K} \cap \SSS_+^n.
\end{equation}
We shall assume that $C$ is positive definite or, equivalently, that $0\in {\rm int} (P)$.
The dual body to our spectrahedron is
written in the coordinates on $\R^m$~as 
\[
 P^\Delta \,\,= \,\, \left\{ \,y\in \R^m  \;|\; \forall \, x\in P \text{ with } \langle y,x\rangle \leq 1\right\}.
\]
We can express $P^\Delta$ as a projection of the
$\binom{n+1}{2}$-dimensional spectrahedron
\begin{equation}
\label{upstairsspec}
Q \,\,= \,\, \{\,U \in \SSS_+^n \,\, | \, \, \langle U, C \rangle \leq 1 \,\}.
\end{equation}
Namely, consider   the linear map dual to the inclusion 
of the linear subspace  $\mathcal{W} = \Span(A_1,A_2,\ldots,A_m)$ 
in the $\binom{n+1}{2}$-dimensional real vector space~$\SSS^n$:
\begin{align*}
\pi_\mathcal{W}: \SSS^n & \, \rightarrow \,\,\, \SSS^n/\mathcal{W}^\perp 
\,\,\simeq \,\,\R^m \\
      U &\,\mapsto \bigl( \langle U, A_1 \rangle, \langle U, A_2 \rangle, \ldots,
      \langle U, A_m \rangle \bigr) .
      \end{align*}   

\begin{remark}
\label{rmk:specshadow}
The convex body $P^\Delta$ dual to the spectrahedron $P$
is affinely isomorphic to the closure of the image of the spectrahedron
(\ref{upstairsspec}) under the linear map $\pi_\mathcal{W}$, i.e. $P^\Delta \simeq \overline{\pi_\mathcal{W}(Q)}$. This result is  due to
Ramana and Goldman \cite{RG}.
In summary, while
the dual to a spectrahedron is generally not a spectrahedron,
it is always a  \emph{spectrahedral shadow}.
See also Theorem \ref{thm:shadowsareclosed}.
\end{remark}

\begin{figure}
\includegraphics[width=0.38\textwidth]{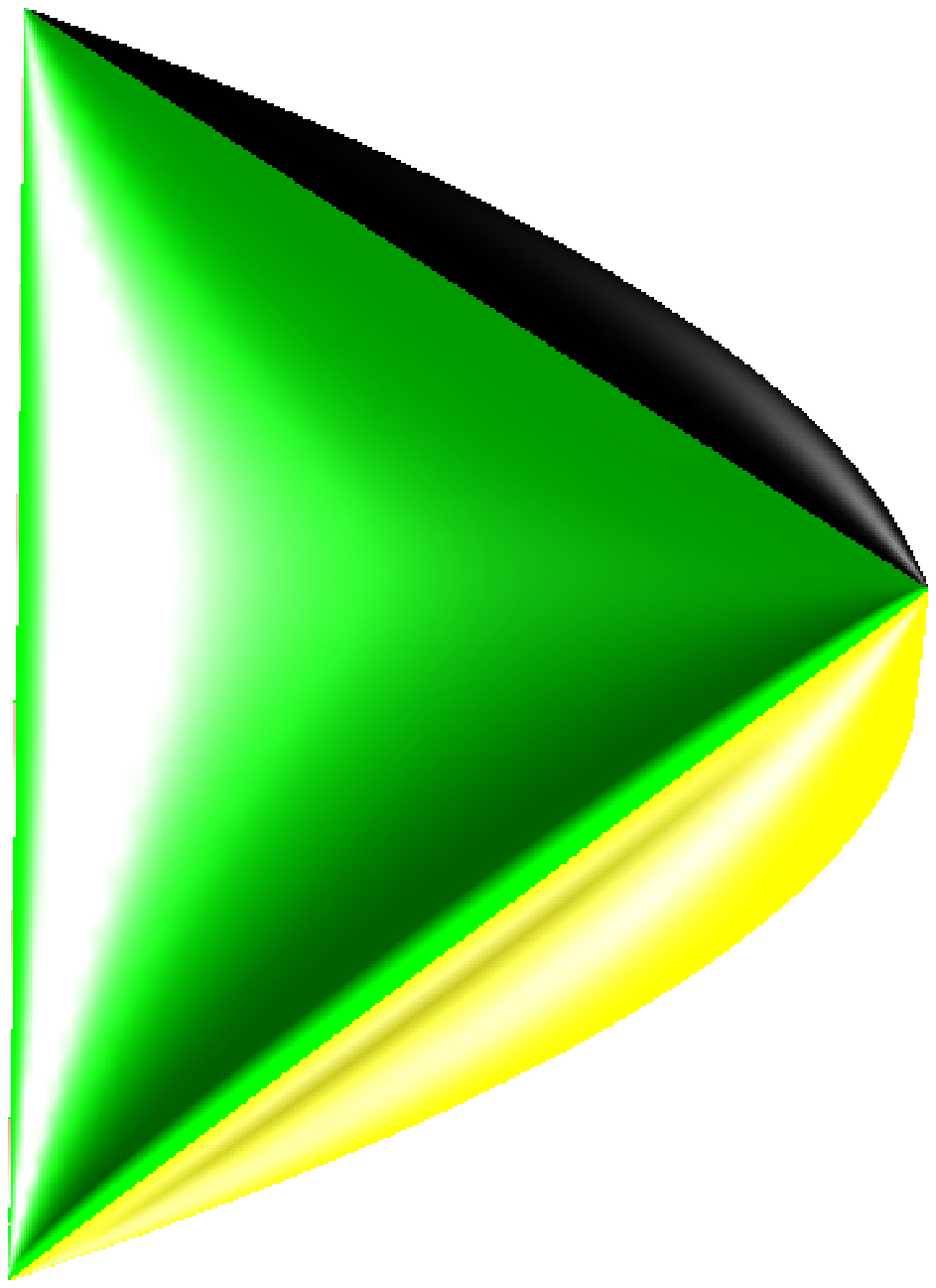}\hfill
\includegraphics[width=0.44\textwidth]{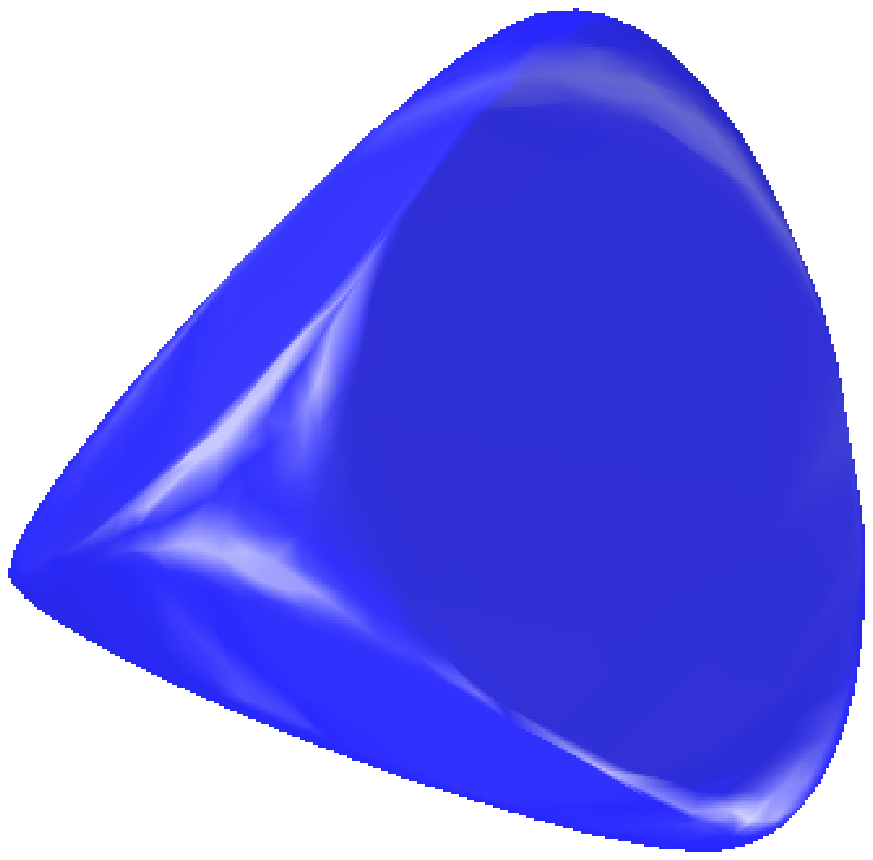}
\caption{The elliptope $P=\mathcal{E}_3$ and its dual convex body $P^\Delta$.\label{fig::samosa}}
\end{figure}

\begin{example} 
\label{ex:elliptope1}
The {\em elliptope} $\mathcal{E}_n$ is the spectrahedron consisting of
all {\em correlation matrices} of size $n$, see \cite{LP}. These are the
positive semidefinite symmetric $n {\times} n$-matrices whose
diagonal entries are~$1$. We consider the
case $n=3$:
\begin{equation}
\label{E3}
\mathcal{E}_3 \,\,= \,\, \left\{(x,y,z)\in \R^3 \, \,\bigg|\,\, 
\begin{pmatrix}1&x&y\\x&1&z\\y&z&1\end{pmatrix} \succeq 0\right\}.
\end{equation}
This spectrahedron of dimension $m=3$
is shown on the left in Figure~\ref{fig::samosa}. 
The algebraic boundary of $\mathcal{E}_3$ is the cubic surface $X$
defined by the vanishing of the $3 \times 3$-determinant in (\ref{E3}).
That surface has four isolated singular points
\[
X_{\rm sing} = \{(1,1,1),(1,-1,-1),(-1,1,-1),(-1,-1,1)\}.
\]
The six edges of the tetrahedron ${\rm conv}(X_{\rm sing})$
are edges of the elliptope $\mathcal{E}_3$.
The dual body, 
shown on the right of Figure~\ref{fig::samosa}, is the spectrahedral shadow
\begin{equation}
\label{E3dual}
\mathcal{E}_3^\Delta \,\,= \,\,
\left\{ (a,b,c) \in \R^3 \,\,\bigg|\,\, 
\exists \,u,v \in \R \,:\,
\begin{pmatrix}\,u&a&b\\ \,a&v&c\\ \, b&c&2{-}u{-}v\end{pmatrix} \succeq 0\right\}.
\end{equation}
The algebraic boundary of $\mathcal{E}_3^\Delta$ can be computed
by the following method. We
form the ideal generated by the determinant in (\ref{E3dual})
and its derivatives with respect to $u$ and $v$, and we eliminate 
$u,v$.
This results in the polynomial
\[ 
(a^2b^2+b^2c^2+a^2c^2-2abc)(a+b+c+1)(a-b-c+1)(a-b+c-1)(a+b-c-1). 
\]
The first factor is the equation of the {\em Steiner quartic} surface
$X^*$ which is dual to {\em Cayley cubic} surface $X  = \partial_a \mathcal{E}_3$.
The four linear factors represent the arrangement 
$(X_{\rm sing})^*$ of the four planes dual to the four singular points.

Thus the algebraic boundary of the dual body $\mathcal{E}_3^\Delta$ is
the reducible surface
\begin{equation}
\label{dualtoE3}
\partial_a \mathcal{E}_3^\Delta  \,\,= \,\, X^* \,\cup \, (X_{\rm sing})^* \quad \subset \,\,\,(\PP^3)^*.
\end{equation}
We note that  $\mathcal{E}_3^\Delta$ is not a spectrahedron
as it fails to be a {\em basic semi-algebraic set}, i.e.~no
polynomial $\phi$ satisfies
$\,\mathcal{E}_3^\Delta = \{ (a,b,c) \in \R^3 \,: \,\phi(a,b,c) \geq 0 \} $.
This is seen from  the fact
that the Steiner surface intersects the interior of $\mathcal{E}_3^\Delta$. 
\qed
\end{example}

Semidefinite programming (SDP) is the branch of convex optimization
that is concerned with maximizing a linear function $b$ over a spectrahedron:

\begin{equation}
\label{eqn::primalConic}
  p^* \,\, := \,\, \underset{x}{\text{Maximize}}  \,\,
  \langle b, x \rangle\\
\,\,\,  \text{subject to}  \,\,\,  x \in P.
  \end{equation}
Here $P$ is as in (\ref{spectraP}).
As the semidefiniteness of a matrix 
is equivalent to the 
simultaneous non-negativity of its principal minors, 
SDP is an instance of 
the polynomial optimization problem
 \eqref{eq::nonlinearOPT}. Lagrange duality theory applies
here by
   \cite[\S5]{BV}. 
   We shall derive  the optimization problem dual to (\ref{eqn::primalConic})~from
\begin{equation}
\label{eqn::dualConic}
  d^* \,\, := \,\, \underset{\lambda}{\text{Minimize}}
  \,\, \lambda
\,\,\, \text{subject to}  \,\,\,
   \frac{1}{\lambda} b \,\in \, P^\Delta\,.
\end{equation}
Since we assumed $0 \in {\rm int} (P)$, strong duality holds and we have $p^*=d^*$.

The fact that $P^\Delta$ is a spectrahedral shadow implies that
the dual optimization problem is again a semidefinite optimization problem.
In light of Remark \ref{rmk:specshadow}, the condition 
$\frac{1}{\lambda} b \in P^\Delta$ can be expressed as follows:
\[
\exists \ U \,: \,U \succeq 0\,,\,\, \langle C, U \rangle \leq 1 \, \,
\hbox{and} \,\,\,
b_i = \lambda \langle A_i,  U \rangle
\,\,\, \hbox{for}\,\,\, i = 1,2,\ldots,m . 
\]
Since the optimal value of \eqref{eqn::dualConic}
is attained at the boundary of $P^\Delta$,  we can here replace
the condition  $\langle C,U \rangle \leq 1$ 
with $\langle C,U \rangle = 1$. This is in fact what was 
done to obtain (\ref{E3dual}).
If we now set   $Y = \lambda U$, then  \eqref{eqn::dualConic} translates into 
\begin{equation}
\label{eqn::dualConicSDP}
\begin{aligned}
&  d^* \,:= \,\underset{Y \in \SSS^n_+}{\text{Minimize}}& &\langle C,Y\rangle\\
& \text{subject to} & & \langle A_i, Y \rangle \,=\, b_i 
\,\,\, \hbox{for} \,\,\, i=1, \ldots, m\\
& & & \hbox{and} \,\,\, Y \succeq 0.
\end{aligned}
\end{equation}
We recall that $\mathcal{W} = \Span(A_1,A_2,\dots,A_m)$  and
we fix any matrix $B \in \SSS^n$ with 
$\langle A_i , B \rangle = b_i$ for $i=1,\ldots,m$.
Then (\ref{eqn::dualConicSDP}) can be written as follows:
\begin{equation}
\label{eqn::dualConicLinSpace}
  d^* \,:=\,\underset{Y \in \SSS^n_+}{\text{Minimize}} \,\langle C, Y\rangle \,\,\,
  \text{subject to} \,\,\,
   Y \in (B + \mathcal{W}^\perp ) \cap \SSS^n_+ 
\end{equation}
The following reformulation of (\ref{eqn::primalConic}) 
highlights the symmetry between the primal and dual formulations
of our semidefinite programming problem:
\begin{equation}
\label{eqn::ConicLinSpace}
  p^* \,:=\,\underset{X \in \SSS^n_+}{\text{Maximize}} \,\langle B, C- X \rangle
\,\,  \text{subject to} \,\,   X \in (C + \mathcal{W} ) \cap \SSS^n_+ 
\end{equation}
Then the following variant of the Karush-Kuhn-Tucker Theorem holds:

\begin{thm} {\rm \cite[\S 5.9.2]{BV}}
If both the primal problem (\ref{eqn::ConicLinSpace})
and its dual (\ref{eqn::dualConicLinSpace})
are strictly feasible, then the KKT conditions take the following~form:
\begin{align*}
X &\in ( C + \mathcal{W} ) \cap \SSS^n_+\\
Y &\in (B +  \mathcal{W}^\perp) \cap \SSS^n_+ \\
X \cdot Y &= 0 \quad \text{(complementary slackness)}.
\end{align*}
These conditions characterize all the pairs $(X,Y)$ of optimal solutions.
\end{thm}
This theorem can be related to the general optimality conditions~\eqref{eq::KKT} 
by regarding the entries of $Y\in \SSS^n$ as the (Lagrangian) dual variables to the positive semidefinite constraint $X = C - \sum_{i=1}^m x_i A_i
\succeq 0$. 
The three conditions are both necessary and sufficient since semidefinite programming is a convex problem and every local optimum is also a global optimal solution.

In order to study algebraic and geometric properties of SDP, we will relax the conic
inequalities $\,X,Y  \in \mathcal{S}_+^n \,$ and focus only on the
{\em KKT equations}
\begin{equation}
\label{eq:KKTsdp}
 X \in  C + \mathcal{W} \, , \,\,
Y \in B +  \mathcal{W}^\perp \,\,\, \hbox{and} \,\,\,
X \cdot Y = 0 .   
\end{equation}
Given $B,C$ and $\mathcal{W}$, our problem is to solve 
the polynomial equations (\ref{eq:KKTsdp}). The theorem ensures
that, among its solutions $(X,Y)$, there is precisely one
pair of positive semidefinite matrices. That pair is the one desired in SDP.

\begin{example}
\label{ex:elliptope2}
Consider the problem of
minimizing  a linear function $Y \mapsto \langle C,Y\rangle$ over the set
 of all correlation matrices $Y$, that is, over the elliptope $\mathcal{E}_n$
 of Example \ref{ex:elliptope1}.
 Here $\,m=n$, $B$ is the identity matrix, 
 $\mathcal{W}$ is the space of all diagonal matrices, 
 and  $\mathcal{W}^\perp$ consists of matrices with zero diagonal.
 The  dual problem is to maximize the trace of $C-X$ over all matrices $X \in \SSS^n_+$ such that $C-X$ is diagonal.
 Equivalently, we seek to find the minimum trace $t^*$
  of any positive semidefinite matrix that agrees with $C$
 in its off-diagonal~entries.
 
 For $n=4$, the KKT equations (\ref{eq:KKTsdp}) can be written in the form
 \begin{equation}
 \label{E4KKT}
 X \cdot Y \,\,\, = \,\,\,
 \begin{pmatrix}
  x_1 &   c_{12} & c_{13} & c_{14} \\
 c_{12} &  x_2 & c_{23}  & c_{24} \\
 c_{13} & c_{23} & x_3 & c_{34} \\
 c_{14} & c_{24} & c_{34} & x_4 
 \end{pmatrix} \cdot
 \begin{pmatrix}
 1 & y_{12} & y_{13} & y_{14} \\
 y_{12} & 1 & y_{23} & y_{24} \\
 y_{13} & y_{23} & 1 & y_{34} \\
 y_{14} & y_{24} & y_{34} & 1 
 \end{pmatrix} \,\,\, = \,\,\, 0 .
 \end{equation}
 This is a system of $10$ quadratic equations in $10$ unknowns.
For general values of the $6$ parameters $c_{ij}$, these equations
have $14$ solutions. Eight of these solutions
have ${\rm rank}(X) = 3 $ and ${\rm rank}(Y) = 1$
and they are defined over $\Q(c_{ij})$.
The other six solutions form an irreducible
variety over $\Q(c_{ij})$ and they satisfy
${\rm rank}(X) = {\rm rank}(Y) = 2$.
This case distinction reflects the boundary structure
of the dual body to the six-dimensional elliptope $\mathcal{E}_4$:
\begin{equation}
\label{dualtoE4}
\partial_a \mathcal{E}_4^\Delta \quad = \quad
\{ {\rm rank}(Y) \leq 2 \}^* \,\cup \,
\{ {\rm rank}(Y) = 1 \}^* . \quad
\end{equation}
Indeed, the boundary of $\mathcal{E}_4$ is the quartic hypersurface
$\{ {\rm rank}(Y) \leq 3 \}$,
its singular locus is the degree $10$ threefold
$\{ {\rm rank}(Y) \leq 2\}$, 
and, finally, the singular locus of that threefold
consists of eight matrices of rank one:
\[
 \{ {\rm rank}(Y) = 1 \} \,\, = \,\,
\bigl\{ \,(u_1,u_2,u_3,u_4)^T \cdot (u_1,u_2,u_3,u_4) \,\,:\,\, u_i \in \{-1,+1 \} \,\bigr\}. 
\]
The last two strata are dual to the hypersurfaces in (\ref{dualtoE4}).
The second component in (\ref{dualtoE4}) consists of eight hyperplanes,
while the first component is irreducible of degree $18$. 
The corresponding  projective hypersurface is defined by an  irreducible homogeneous polynomial
of degree $18$ in seven unknowns $c_{12},c_{13}, c_{14}, c_{23}, c_{24},c_{34}, t^*$.
That polynomial has degree $6$  in the special unknown $t^*$.
Hence, the algebraic degree of our SDP is $6$ when ${\rm rank}(Y) = 2$.
\qed
\end{example}

In algebraic geometry, it is natural to regard 
the matrix pairs $(X,Y)$ as points in the product of projective spaces
$\PP(\SSS^n) \times \PP(\SSS^n)^*$. This has the advantage that
 solutions of (\ref{eq:KKTsdp}) are invariant under scaling, i.e. whenever $(X,Y)$ is a solution, then so is $(\lambda X,\mu Y)$ for any nonzero $\lambda,\mu \in \R$. In that setting, there are no worries about complications due to solutions at infinity.

For the algebraic formulation we assume that,
without loss of generality, 
\[
b_1=1,b_2=0,b_3=0,\dots,b_m=0.
\]
This means that $\langle A_1, X \rangle = 1$
plays the role of the homogenizing variable.
Our SDP instance is specified by two linear subspaces
 of symmetric matrices:
 \[
 \mathcal{L} = \Span(A_2,A_3,\dots,A_m) \,\,\subset \,\,
\mathcal{U} = \Span(C,A_1,A_2,\dots,A_m) \,\,\subset \,\, \SSS^n.
\]
Note that we have the following identifications:
\[
 \R C + \mathcal{W} \,=\, \mathcal{U} \quad \hbox{and} \quad
\R B + \mathcal{W}^\perp \, = \, \R B + (\mathcal{L}^\perp \cap A_1^\perp)
\, = \, \mathcal{L}^\perp. 
\]
 With the linear spaces $\mathcal{L} \subset \mathcal{U}$, we  write the
  {\em homogeneous KKT equations} as
\begin{equation}
\label{eq:KKThomog}
X \in \mathcal{U},\; Y \in \mathcal{L}^\perp \text{ and } X \cdot Y = 0.
 \end{equation}
Here is an {\em abstract definition of semidefinite programming}
that might appeal to some of our readers:
Given any flag of linear subspaces
$\mathcal{L} \subset \mathcal{U} \subset \SSS^n$ 
with ${\rm dim}(\mathcal{U}/\mathcal{L}) = 2$,
locate the unique semidefinite point in the
variety~(\ref{eq:KKThomog}).
For instance, in  Example \ref{ex:elliptope2}
the space $\mathcal{L}$ consists of traceless diagonal matrices
and $\mathcal{U}/\mathcal{L} $ is spanned by
the unit matrix $B $ and one off-diagonal matrix $C$.
We seek to solve the matrix equation $X \cdot Y = 0$
where the diagonal entries of $X$ are constant
and the off-diagonal entries of $Y$ are proportional to $C$.

The formulation (\ref{eq:KKThomog}) suggests that we study the variety
$\{XY = 0\}$ for pairs of symmetric matrices $X$ and $Y$.
  In~\cite[Eqn.~(3.9)]{NRS} it was shown that this variety has the following
  decomposition into irreducible components:
\[
\{XY=0\} \,\,= \,\, \bigcup_{r=1}^{n-1} \{XY = 0\}^r \quad
\subset \,\,\, \PP(\SSS^n) \times \PP(\SSS^n)^*.
\]
Here $\{XY = 0\}^r$ denotes the subvariety consisting of pairs $(X,Y)$
where ${\rm rank}(X) \leq r$ and ${\rm rank}(Y) \leq n-r$.
This is irreducible because, by Example~\ref{ex::conormal_SDP},
 it is   the conormal variety of the variety of symmetric matrices of rank $\leq r$. 
The KKT equations describe sections of these conormal varieties:
\begin{equation}
\label{eqn::projKKT}
\{XY = 0\}^r \,\cap \, \bigl(\,\PP(\mathcal{U}) \times \PP(\mathcal{L}^\perp)\,\bigr).
\end{equation}
All solutions of a semidefinite 
optimization problem (and thus also the boundary of a spectrahedron and its dual) can be characterized by rank conditions. The main result in~\cite{NRS} describes the case when the section in \eqref{eqn::projKKT} is generic:

\begin{thm}{\rm \cite[Theorem 7]{NRS}}
\label{thm:thm7}
For generic subspaces $\mathcal{L} \subset \mathcal{U} \subset \SSS^n$ with
${\rm dim}(\mathcal{L}) = m-1$ and ${\rm dim}(\mathcal{U}) = m+1$,
the variety \eqref{eqn::projKKT} is empty unless 
\begin{equation}
\label{eq:pataki}
 \binom{n-r+1}{2} \leq m \quad \text{and} \quad \binom{r+1}{2}\leq \binom{n+1}{2}-m  .
 \end{equation}
In that case, the variety \eqref{eqn::projKKT} is reduced, nonempty and zero-dimensional and at each point the rank of $X$ and $Y$ is $n-r$ and $r$ respectively (strict complementarity). The cardinality of this variety depends only on $m$, $n$ and~$r$.
\end{thm}

The generic choice of subspaces $\mathcal{L} \subset \mathcal{U}$
corresponds to the assumption that our matrices $A_1,A_2,\ldots,A_m, B, C$
lie in a certain dense open subset in the space of all SDP instances.
The inequalities (\ref{eq:pataki}) are known as {\em Pataki's inequalities}.
If $m$ and $n$ are fixed then they give a lower bound and an upper bound
for the possible ranks $r$ of the optimal matrix of a generic SDP instance.
The variety \eqref{eqn::projKKT} represents all complex solutions of the
KTT equations for such a generic SDP instance. Its
cardinality, denoted $\,\delta(m,n,r)$, is known as the
{\em algebraic degree of semidefinite programming}.

\begin{cor}
Consider the variety of symmetric $n {\times} n$-matrices
of rank $\leq r$ that lie in the generic $m$-dimensional linear subspace
$\PP(\mathcal{U})$ of $\PP(\SSS^n)$. Its dual variety 
is a hypersurface if and only if Pataki's inequalities (\ref{eq:pataki}) hold, and 
the degree of that hypersurface is $\delta(m,n,r)$, the algebraic degree of SDP.
\end{cor}

\begin{proof} The genericity of $\mathcal{U}$ ensures that
$\{ XY = 0\}^r \,\cap \,(\,\PP(\mathcal{U}) \times \PP(\mathcal{U})^*)\,$ is
the conormal variety of the given variety. We obtain its
dual  by projection onto the second factor
$\PP(\mathcal{U})^* = \PP(\SSS^n/\mathcal{U}^\perp)$. The
degree of the dual hypersurface is found by intersecting with a generic line.
The line we take is $\PP(\mathcal{L}^\perp/\mathcal{U}^\perp)$.
That intersection corresponds to the second factor
$\PP(\mathcal{L}^\perp)$ in (\ref{eqn::projKKT}).
\end{proof}

We note that the symmetry in the equations
(\ref{eq:KKThomog}) implies the duality 
\[
 \delta \bigl( m,n,r \bigr) \quad = \quad \delta \bigl(
 \binom{n+1}{2} - m, n, n-r \bigr), 
\]
first shown in \cite[Proposition~9]{NRS}.
See also \cite[Table 2]{NRS}.
Bothmer and Ranestad \cite{BR} derived an explicit
combinatorial formula for the algebraic degree of SDP.
Their result implies that $\delta(m,n,r)$ is a polynomial
of degree $m$ in $n$ when $n-r$ is fixed. 
For example, in addition to \cite[Theorem 11]{NRS}, we have
\[
 \delta(6,n,n-2) \quad = \quad 
\frac{1}{72}
 \bigl(11 n^6 - 81 n^5 + 185 n^4 - 75 n^3 - 196 n^2 + 156 n \bigr).
\]
 
The algebraic degree of SDP represents a universal upper bound on the
intrinsic algebraic complexity of optimizing a linear function over any $m$-dimensional
spectrahedron of $n {\times} n$-matrices.
The algebraic degree can be much smaller for families of instances
involving special matrices $A_i$, $B$ or $C$.

\begin{example}
Fix $n=4$ and $m= 6 =  {\rm dim}(\mathcal{E}_4) $.
Pataki's inequalities (\ref{eq:pataki}) state that
the rank of the optimal matrix is  $r = 1$ or $r=2$, and this was indeed
observed in Example  \ref{ex:elliptope2}. 
For $r=2$ we had found the algebraic degree six
when solving (\ref{E4KKT}). However, here
$B$ is the identity matrix and $A_1, A_2,A_3,A_4$ are diagonal.
When these are replaced by generic symmetric matrices, then
the algebraic degree jumps from six to 
$\, \delta(6,4,2) \, = \, 30 $.  \qed
\end{example}

We now state a result that elucidates the
decompositions in
(\ref{dualtoE3}) and~(\ref{dualtoE4}).

\begin{thm}
\label{boundarydualspec}
If the matrices $A_1,\ldots,A_m$ and $C$
in the definition (\ref{spectraP}) of the spectrahedron $P$ are sufficiently generic, then
the algebraic boundary of the dual body $P^\Delta$ is
the following union of dual hypersurfaces:
\begin{equation}
\label{eq:boundarydualspec}
 \partial_a P^\Delta \,\quad \subseteq \,\, \bigcup_{r \,{\rm as}  \,{\rm in} \,
(\ref{eq:pataki})} \!\! \{X \in \mathcal{L} \,\,|\,\,{\rm rank}(X) \leq r \}^* \end{equation}
\end{thm}

\begin{proof}
Let $\mathcal{Y}$ be any irreducible component of $\partial_a P^\Delta \subset (\PP^m)^*$.
Then $\mathcal{Y} \cap \partial P^\Delta$ is a semi-algebraic subset of
codimension $1$ in $P^\Delta$. We consider a general point in that set.
The corresponding hyperplane $H$ in the primal $\R^m$ supports
the spectrahedron $P$ at a unique point $Z$.
Then $r = {\rm rank}(Z)$ satisfies Pataki's inequalities, by
Theorem \ref{thm:thm7}. Moreover, the genericity in our choices of
$A_1, \ldots, A_m, C, H$ ensure that $Z$ is a regular point in
$\{X \in \mathcal{L} \,|\, {\rm rank}(X) \leq r \}$.
Bertini's Theorem ensures that this determinantal variety is irreducible
and that its singular locus consists only of matrices of rank $< r$.
This implies that
$\{X \in \mathcal{L} \,|\, {\rm rank}(X) \leq r \}$
is the Zariski closure of
$\{X \in P \,|\, {\rm rank}(X) = r \}$, and hence also of
a neighborhood of $Z$ in that rank stratum.
Likewise, $\mathcal{Y}$ is the Zariski closure in $(\PP^m)^*$ of 
$\mathcal{Y} \cap \partial P^\Delta$.
An open dense subset of points in $\mathcal{Y} \cap \partial P^\Delta$
corresponds to hyperplanes that support $P$ at
a rank $r$ matrix. We conclude 
$\,\mathcal{Y}^* = \{X \in \mathcal{L} \,|\, {\rm rank}(X) \leq r \}$.
Biduality completes the proof.
\end{proof}

Theorem \ref{boundarydualspec} is similar to 
Theorem \ref{boundarytheorem} in that it characterizes
the algebraic boundary in terms of dual hypersurfaces.
Just like in Section~\ref{sec::convex_hull}, we can apply this result
to compute  $\partial_a P^\Delta$. For each
rank $r$ in the Pataki range
(\ref{eq:pataki}), we need to check whether 
the corresponding dual hypersurface 
meets the boundary of $P^\Delta$. 
The indices $r$ which survive this test
determine $\partial_a P^\Delta$.

When the data that specify the spectrahedron
$P$ are not generic but special then the
computation of $\partial_a P^\Delta$ is more subtle
and we know of no formula as simple as
(\ref{eq:boundarydualspec}).
This issue certainly deserves further research.

We close this section with an interesting 
$3$-dimensional example.

\begin{example}
The {\em cyclohexatope} is a spectrahedron
with $m=3$ and $n=5$ that arises in the study
of chemical conformations \cite{GoS}. Consider the
following {\em Sch\"onberg matrix} for 
the pairwise distances $\sqrt{D_{ij}}$ among six carbon atoms:
\begin{small}
\[ \! \begin{pmatrix}
2 D_{12} &  
 D_{12}{+}D_{13} {-} D_{23} \!&\!
D_{12}{+}D_{14} {-} D_{24} \!&\!
D_{12}{+}D_{15} {-} D_{25} \!&\!
D_{12}{+}D_{16} {-} D_{26} \\
   D_{12}{+}D_{13} {-} D_{23}\! &\!
   2 D_{13} \!&\!
D_{13}{+}D_{14} {-} D_{34} \! &\!
D_{13}{+}D_{15} {-} D_{35} \!&\!
D_{13}{+}D_{16} {-} D_{36} \\
 D_{12}{+}D_{14} {-} D_{24} \!&\!
 D_{13}{+}D_{14} {-} D_{34} \!&\!
   2 D_{14} \! &\!
D_{14}{+}D_{15} {-} D_{45} \! &\!
D_{14}{+}D_{16} {-} D_{46} \\
 D_{12}{+}D_{15} {-} D_{25} \! &\!
 D_{13}{+}D_{15} {-} D_{35} \! &\!
 D_{14}{+}D_{15} {-} D_{45} \! &\!
   2 D_{15} \! &\!
   D_{15}{+}D_{56} {-} D_{56} \\
 D_{12}{+}D_{16} {-} D_{26} \! &\!
 D_{13}{+}D_{16} {-} D_{36} \! &\!
 D_{14}{+}D_{16} {-} D_{46} \! &\!
   D_{15}{+}D_{56} {-} D_{56} \! &\!
   2 D_{16}
\end{pmatrix}
\]
\end{small}
The $D_{ij}$ are the squared distances among six points in $\R^3$
if and only if this matrix is positive-semidefinite of rank $\leq 3$.
 The points represent the carbon atoms in 
{\em cyclohexane} $C_6 H_{12}$ if and only if
$D_{i,i+1} = 1$ and $D_{i,i+2} = 8/3$ for all 
indices $i$, understood cyclically.
The three diagonal distances are unknowns, so,
for cyclohexane conformations, the above Sch\"onberg matrix equals
\[
{\bf C}_6(x,y,z) \quad = \quad
 \begin{pmatrix}
2  &  
 8/3   & 
x  - 5/3  & 
11/3 - y  & 
- 2/3 \\
   8/3   & 
   2   & 
5/3 + x    & 
8/3  & 
11/3 -  z \\
x - 5/3  & 
 5/3 + x   & 
   16/3   & 
x + 5/3    & 
x - 5/3 \\
 11/3 - y   & 
 8/3   & 
 x + 5/3    & 
   2 y   & 
   8/3 \\
 - 2/3   & 
 11/3 -  z   & 
 x -  5/3   & 
   8/3   & 
   16/3
\end{pmatrix}.
\]
The {\em cyclohexatope} ${\rm Cyc}_6$ is the spectrahedron in $\R^3$ defined
by $\,{\bf C}_6(x,y,z) \succeq 0 $.
Its algebraic boundary decomposes as
$\,\partial_a {\rm Cyc}_6 =  V(f) \cup V(g)$, where
\[
 \begin{matrix}
   f & = & 27 x y z-75 x-75 y-75 z-250  \qquad \quad  {\rm and} \\
  g & = & 3 x y+3 x z+3 y z-22 x-22 y-22 z+121. \\
 \end{matrix}
\]
The conformation space of cyclohexane is the real algebraic variety
\[
\bigl\{(x,y,z) \in {\rm Cyc}_6 \,\,|\,\, {\rm rank}({\bf C}_6(x,y,z)) \leq 3\, \bigr\} 
\quad =  \quad V(f,g) \,\,\cup \,\,V(g)_{\rm sing} .
\]
The first component is the closed curve of all {\em chair conformations}.
The second component is the {\em boat conformation} point
 $\,(x,y,z) = \bigl( \frac{11}{3},  \frac{11}{3},  \frac{11}{3} \bigr) $.
 These are well-known to chemists~\cite{GoS}.
Remarkably, the cyclohexatope coincides with the convex hull of these two components.
This spectrahedron is another example of a 
convex hull of a space curve, now with an isolated point.
Semidefinite programming over the cyclohexatope
means computing the conformation which minimizes a linear function in the
squared distances $D_{ij}$. \qed
\end{example}

\section{Spectrahedral Shadows}
\label{sec::spec_shadow}
A {\em spectrahedral shadow} is the image of a spectrahedron under a linear map.
The class of spectrahedral shadows is much larger than the class of spectrahedra. In fact, it has even been conjectured that every convex basic semialgebraic set in $\R^n$ is a spectrahedral shadow~\cite{HN}.
Our point of departure is the result, known to optimization experts, that
 the convex body dual to
a spectrahedral shadow is again a spectrahedral shadow
\cite[Proposition 3.3]{GN}.

\begin{thm} \label{thm:shadowsareclosed}
The class of spectrahedral shadows is closed under duality.
\end{thm}

\begin{proof}[Construction]
A spectrahedral shadow can be written in the form
\[
P \,\,= \,\,\, \left\{ x\in \R^m \,\big|\; \exists\, y\in \R^p \text{ with } C + \sum_{i=1}^m x_i A_i + \sum_{j=1}^p y_j B_j\succeq 0 \right\}.
 \]
 An expression for the dual body $P^\Delta$ is obtained by
 the following variant of
  the construction in Remark~\ref{rmk:specshadow}.
We consider the same linear map as before:
  $$
\pi : \SSS_+^n \rightarrow \R^m,\,\,
U \mapsto (\langle A_1,U\rangle,\dots,\langle A_m,U\rangle).
$$
 We apply this linear map $\pi$ to the spectrahedron
  \[
Q \,\, = \,\, \left \{U \in \SSS^n_+ \,\,|\,\, \langle C,U \rangle \leq 1 \text{ and }
  \langle B_1,U \rangle = \cdots   = \langle B_p \,,U \rangle = 0
   \right\}.
\]
The closure of the spectrahedral shadow $\pi(Q)$ equals
the dual convex body $ P^\Delta$.
This closure is itself a spectrahedral shadow, by
 \cite[Corollary 3.4]{GN}.
\end{proof}

We now consider the following problem: Given a real variety
$X \subset \R^n$, find a representation of its convex hull
${\rm conv}(X)$ as a spectrahedral shadow.
A systematic approach to computing such representations
was introduced by Lasserre \cite{Las}, and further
developed by Gouveia {\it et al.}~\cite{GPT}.
It is based on the relaxation of non-negative
polynomial functions on $X$ as sums of squares in
the coordinate ring $\R[X]$.
This approach is known as {\em moment relaxation},
in light of the duality between positive polynomials and
moments of measures.

We shall begin by exploring these ideas for homogeneous polynomials
of even degree $2d$ that are non-negative on $\R^n$. These form a
cone in a real vector space of dimension $N=\binom{d+n-1}{d}$.
Inside that cone lies the smaller {\em SOS cone} of polynomials $p$
that are sums of squares of polynomials of degree~$d$:
\begin{equation}
\label{eq:SOS}
 p \quad = \quad  q_1^2  + q_2^2  + \,\cdots\, + q_N^2.
 \end{equation}
 By Hilbert's Theorem \cite{Mar}, this inclusion of convex cones is strict unless
 $(n,2d) $ equals $(1,2d)$ or $(n,2)$ or $(3,4)$.
 The SOS cone is easily seen to be a spectrahedral shadow.
Indeed, consider an unknown symmetric matrix $Q \in \SSS^N$ and write
 $p = v^T Q v$ where $v$ is the vector of all $N$ monomials of degree $d$.
The matrix $Q$ is positive semidefinite if it has a Cholesky factorization $Q=C^TC$. The resulting identity $\,p =  (C v)^T  (Cv) \,$ can be rewritten as (\ref{eq:SOS}). Hence the SOS cone is the image of $\SSS_+^N$ under the linear map   $\,Q \mapsto v^T Q v$.

Recent work of Nie \cite{Nie} studies the boundaries of our two cones
via computations with the discriminants we encountered at the end of Section~\ref{subsec::proj_dualities}.
 
\begin{prop}[Theorem 4.1 in \cite{Nie}]
\label{prop:nie}
The algebraic boundary of the cone of homogeneous polynomials $p$
of degree $2d$ that are non-negative on $\R^n$ is
given by the discriminant of a polynomial $p$ with unknown coefficients.
This discriminant is the irreducible hypersurface dual to 
the Veronese embedding 
\[
\PP^{n-1} \hookrightarrow \PP^{N-1},\,
(x_1:\dots:x_{n}) \mapsto (x_1^d:x_1^{d-1}x_2:\dots:x_n^d)
\]
The degree of this discriminant is $\, n(2d-1)^{n-1}$.
\end{prop}
\begin{proof}
 The discriminant of $p$ vanishes
 if and only if there exists $x \in \PP^{n-1}$ with
  $p(x) = 0$ and $\nabla p \big|_x=0$. If $p$ is in the boundary
  of the cone of positive polynomials then such a real point $x$
  exists. For the degree formula see~\cite{GKZ}.
   \end{proof}
 
 Results similar to Proposition \ref{prop:nie}
 hold when we restrict to polynomials $p$
 that lie in linear subspaces. This is
  why the {\em $A$-discriminants} $\Delta_A$
  from Section~\ref{subsec::proj_dualities} are relevant.
  We show this for a 
  $2$-dimensional family of polynomials.
  
\begin{figure}
\vskip -0.5cm
\centering\includegraphics[width=0.6\textwidth]{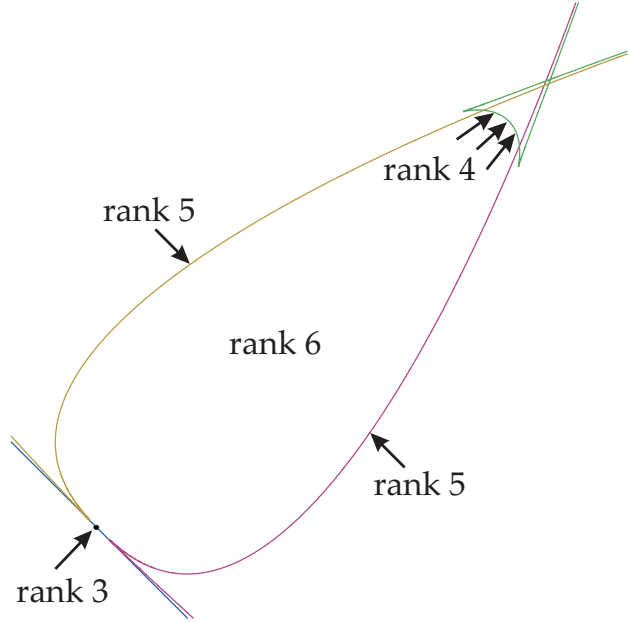}
\vskip -0.3cm
\caption{The discriminant in Example \ref{ex:fab}
defines a curve in the $(a,b)$-plane. 
The {\em spectrahedral shadow} $\mathcal{C}$ is the set of points where
the ternary quartic $f_{a,b}$ is SOS. The  ranks of the corresponding SOS 
matrices $Q$ are indicated.\label{fig:sos2}}
\vskip -0.1cm
\end{figure}  

 \begin{example}
 \label{ex:fab}
 Consider the two-dimensional family of ternary quartics
 \[ f_{a,b}(x,y,z) \quad = \quad x^4 \,+\,y^4 \,+\, ax^3 z\, +\, a y^2 z^2 
 \,+\, b y^3z \,+\, b x^2z^2 \,+ \,(a+b)z^4. \]
Here $a$ and $b$ are parameters. Such a polynomial
is non-negative on $\R^3$ if and only if it is a sum of squares, by
Hilbert's Theorem. This
condition defines a closed convex region $\mathcal{C}$ in the
$(a,b)$-plane $\R^2$. It is non-empty because $\, (0,0) \in \mathcal{C}$. Its
boundary $\partial_a(\mathcal{C})$ is derived from the {\em
$A$-discriminant} $\Delta_A$, where
\[
A \quad = \quad \begin{pmatrix}
4  & 0 & 3 & 0 & 0 & 2 & 0 \\  
0  & 4 & 0 & 2 & 3 & 0 & 0 \\
0  & 0 & 1 & 2 & 1 & 2 & 4 
\end{pmatrix}.
\]
  This $A$-discriminant is an irreducible homogeneous polynomial of
  degree $24$ in the seven coefficients. 
  What we are interested in here is the {\em specialized discriminant}
  which is obtained from $\Delta_A$ by substituting the vector of coefficients
$(1,1,a,a,b,b,a+b)$ corresponding to our polynomial $f_{a,b}$. The
specialized discriminant is an inhomogeneous polynomial of degree $24$ in
the two unknowns $a$ and $b$, and it is no longer irreducible. A computation reveals that
it is the product of four irreducible factors whose degrees are $1$, $5$, $5$ and $13$.

The linear factor equals $a+b$.
The two factors of degree $5$ are
\begin{small}
\[
\begin{matrix}
256 a^2 {-} 27 a^5 {+} 512 a b {+} 144 a^3 b {-} 27 a^4 b {+} 256 b^2 {-} 128 a b^2 {+}
144 a^2 b^2 {-} 128 b^3 {-} 4 a^2 b^3 {+} 16 b^4, \\
  256 a^2 {-} 128 a^3 {+} 16 a^4 {+} 512 a b {-} 128 a^2 b {+} 256 b^2 {+} 144 a^2 b^2 {-}
4 a^3 b^2 {+} 144 a b^3 {-} 27 a b^4 {-} 27 b^5.
\end{matrix}
\]
\end{small}
Finally, the  factor of degree $13$ in the specialized discriminant equals
\begin{small}
\[
\begin{matrix}
&  2916 a^{11} b^2+19683 a^9 b^4+19683 a^8 b^5+2916 a^7 b^6+2916 a^6
  b^7+ 19683 a^5 b^8 \\
 & +19683 a^4 b^9+2916 a^2 b^{11}-11664
  a^{12}-104976 a^{10} b^2- 136080 a^9 b^3-27216 a^8 b^4 \\
 & -225504
  a^7 b^5 -419904 a^6 b^6 -225504 a^5 b^7- 27216 a^4 b^8 -136080 a^3
  b^9  \\ &  -104976 a^2 b^{10}
   -11664 b^{12}+93312 a^{11}+217728 a^{10}
  b+76032 a^9 b^2 \\
 & +1133568 a^8 b^3+1976832 a^7 b^4+891648 a^6
  b^5+891648 a^5 b^6+1976832 a^4 b^7 \\
 & +1133568 a^3 b^8+76032 a^2
  b^9+217728 a b^{10}+93312 b^{11} -241920 a^{10} \\
 & -1368576 a^9
  b-2674944 a^8 b^2-1511424 a^7 b^3 -4729600 a^6 b^4 -9369088 a^5 b^5
  \\
 & -4729600 a^4 b^6-1511424 a^3 b^7-2674944 a^2 b^8 - 1368576 a b^9
  -241920 b^{10} \\
 & +663552 a^9+2949120 a^8 b+10539008 a^7 b^2+
  17727488 a^6 b^3 +9981952 a^5 b^4 \\
 & +9981952 a^4 b^5+17727488 a^3
  b^6+10539008 a^2 b^7+2949120 a b^8 +663552 b^9 \\
 & -2719744
  a^8-8847360 a^7 b-14974976 a^6 b ^2-36503552 a^5 b^3 -56360960 a^4
  b^4 \\
 & -36503552 a^3 b^5-14974976 a^2 b^6- 8847360 a b^7-2719744 b^8
  +4587520 a^7 \\
 & +25821184 a^6 b+52035584 a^5 b^2+ 50724864 a^4  b^3
  {+} 50724864 a^3 b^4 {+}52035584 a^2 b^5 \\
 & +25821184 a b^6+4587520
  b^7-6291456 a^6-31457280 a^5 b -94371840 a^4 b^2 \\
 & -138412032 a^3
  b^3-94371840 a^2 b^4-31457280 a b^5-6291456 b^6 +16777216 a^5 \\
 &
  +50331648 a^4 b+67108864 a ^3 b^2+67108864 a^2 b^3+50331648 a
  b^4+16777216 b^5 \\
 & -16777216 a^4-67108864 a^3 b-100663296 a^2
  b^2-67108864 a b^3-16777216 b^4.
\end{matrix}
\]
\end{small}
The relevant pieces of these four curves in the $(a,b)$-plane are depicted in
Figure~\ref{fig:sos2}. The line $a+b=0$ is seen in  the lower left,
the degree $13$ curve is the swallowtail in the upper right, and the two 
quintic curves form the upper-left and lower-right boundary of the enclosed
convex region $\mathcal{C}$.

For each $(a,b) \in \mathcal{C}$, the ternary quartic
$f_{a,b}$ has an SOS representation
\[
f_{a,b}(x,y,z) \,\, = \,\, (x^2,xy,y^2,xz,yz,z^2)\,\cdot\, Q\,\cdot\,
(x^2,xy,y^2,xz,yz,z^2)^T,
\]
where $Q$ is a positive semidefinite $6 {\times} 6$-matrix.
This identity gives $15$ independent linear constraints which,
together with $Q \succeq 0$, define an $8$-dimensional spectrahedron in the
$(21+2) $-dimensional space of parameters $(Q,a,b)$.
The projection of this spectrahedron onto the $(a,b)$-plane
is our convex region $\mathcal{C}$. This proves that $\mathcal{C}$ is
a spectrahedral shadow.
If $(a,b)$ lies in the interior of $\mathcal{C}$ then the fiber
of the projection is a $6$-dimensional spectrahedron.
If $(a,b)$ lies in the boundary $\partial \mathcal{C}$ then the fiber consists
of a single point. The ranks of these unique matrices are indicated in Figure~\ref{fig:sos2}. 
Notice that $\partial \mathcal{C}$ has three singular points, at which the rank drops
from $5$ to $4$ and $3$ respectively. \qed
\end{example}

We now shift towards a functional analytic point of view. The degree $d$ 
is no longer fixed, and we consider all polynomials, not just homogeneous ones.
Polynomials that are non-negative on $\R^n$ form a convex cone $\mathcal{C}$
in the infinite-dimensional real vector space $\R[x_1,\ldots,x_n]$. Its dual cone
$\mathcal{C}^*$ is the set of all linear functionals 
$\,\R[x_1,\ldots,x_n] \rightarrow \R\,$ that are non-negative on  $\mathcal{C}$.
We consider  functionals that evaluate to $1$ on the constant polynomial $1$.
These are represented by the moments of probability measures $\mu$ on $\R^n$:
\[
y_{\alpha} \,\,=\, \int_{\R^n} \! x^{\alpha} d\mu \qquad
\quad \hbox{for} \,\,\ \alpha \,\in\, \N^n.
\]
Points in $\mathcal{C}^*$ are moment sequences $(y_{\alpha}) \in \R^{\N^n}$ of 
Borel measures $\mu$ on $\R^n$.

This setup allows for an elegant and fruitful interpretation of
Lagrange duality for polynomial optimization problems
(\ref{eq::nonlinearOPT}). To keep the exposition and notation simple,
we restrict ourselves to  the  unconstrained problem 
\begin{equation*}
\begin{aligned}
& \underset{x\in \R^n}{\text{minimize}} & &f(x) =\sum_\alpha f_\alpha x^\alpha
\end{aligned}
\end{equation*}
Here we assume that $f$ is bounded from below, say $f\geq \epsilon$, and $\deg(f) = 2d$.
Our problem is equivalent to finding the best possible lower bound:
\begin{equation}
\label{eq::POP}
\begin{aligned}
& \underset{t\in \R}{\text{maximize}} & &t \\
& \text{subject to}
& & f(x)-t\geq 0\,\, \text{ for all }\, x\in \R^n.
\end{aligned}
\end{equation}

The Lagrange dual of the problem~\eqref{eq::POP} reads
\begin{equation*}
\begin{aligned}
& \underset{\mu \in \mathcal{P}}{\text{minimize}} & &\int_{\R^n} f(x) d\mu,
\end{aligned}
\end{equation*}
where $\mathcal{P}$ is the
 convex set of all Borel probability measures on $\R^n$.
We can rewrite this now
as an infinite-dimensional linear optimization problem:
\begin{equation}
\label{eq::POPdual}
\begin{aligned}
& \underset{y \in \mathcal{Y}}{\text{minimize}} & &\sum_{\alpha} f_{\alpha} y_{\alpha} \\
& \text{where}  & &
\mathcal{Y} \, := \, \left\{y\in \R^{\N^n} \,\Big |\,y_{\alpha} = \int_{\R^n} x^{\alpha} d\mu \text{ with } \mu \in \mathcal{P} \right\}.
\end{aligned}
\end{equation}

The two dual problems  \eqref{eq::POP}  and \eqref{eq::POPdual} are as difficult to solve as our original optimization problem.
There is a natural relaxation which is easier, and we can express this
either on the primal side or on the dual side. In~\eqref{eq::POP} 
we replace the constraint that $f(x) - t $ be non-negative on $\R^n$
with the easier constraint that $f(x)-t$ be a sum of squares.
We relax the dual  \eqref{eq::POPdual} by enlarging the
convex set $ \mathcal{Y}$  to the {\em infinite-dimensional spectrahedron}
consisting of all positive semidefinite
moment matrices
\[
M(y) \,=\, \begin{pmatrix}y_{\alpha+\beta}\end{pmatrix}_{\alpha,\beta \in \N^n}\succeq 0.
\]
These two relaxations are again related by Lagrange duality, but now they
represent a dual pair of semidefinite programs. Of course, when we solve
such an SDP in practise, we always restrict to a finite submatrix of $M(y)$,
usually that indexed by all monomials $x^\alpha, x^\beta$ of
some bounded degree $\leq d$.
The question of when such a relaxtion is exact and, if not, how large the
gap can be, is an active area of research in convex algebraic geometry
\cite{GPT, Lasbook, Sch}.

\smallskip

\begin{figure}
\centering\includegraphics[width=0.7\textwidth]{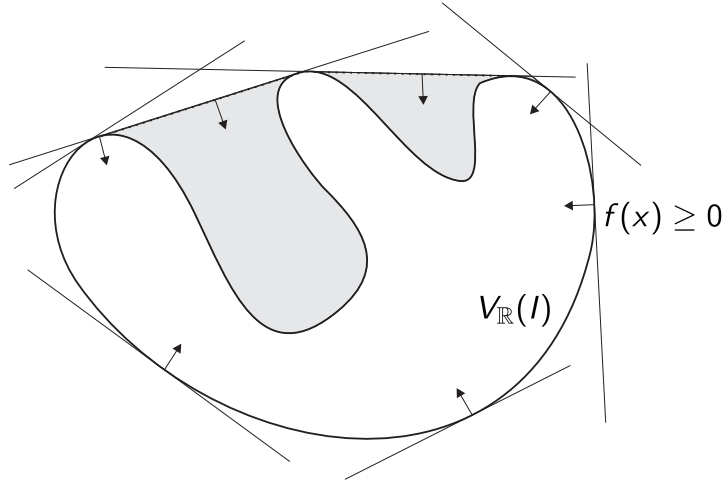}
\caption{Convex hull as intersection of half spaces.\label{fig::suphyp}}
\end{figure}

We now turn our attention to a variant of the above procedure which
approximates the convex hull of a variety by a nested
family of spectrahedral shadows. Let $I$
be an ideal in $ \R[x_1,\ldots,x_n]$ and $ V_\R(I)$ its variety  in $\R^n$.
Consider the set of affine-linear polynomials 
that are non-negative on $V_\R(I)$:
\[
{\rm NN}(I) \quad = \quad \{ \,f \in \R[x_1,\ldots,x_n]_{\leq 1}\,\,|\,\, f(x) \geq 0 \text{ for all } x\in V_\R(I) \}.
\]
In light of the biduality theorem for convex sets (cf.~Section~\ref{subsec::convbody}), 
we can characterize the (closure of) the convex hull of
our variety as follows:
\[
\overline{\conv(V_\R(I))} \quad = \quad \{x\in \R^n\, | \,f(x)\geq 0 \text{ for all } f \in {\rm NN}(I) \}.
\]
The geometry behind this formula is shown in Figure~\ref{fig::suphyp}.

Following Gouveia {\it et al.}~\cite{GPT}, we now replace the
hard constraint that $f(x)$ be non-negative on $V_\R(I)$
with the (hopefully easier) constraint that $f(x)$ be a sum of squares
in the coordinate ring $\R[x_1,\ldots,x_n]/I$. Introducing
a parameter $d$ that indicates the degree of the polynomials
allowed in that SOS representation, we consider the
following set of affine-linear polynomials:
\[ {\rm SOS}_d(I) \,\, = \,\, \bigl\{ \,f  \,\,|\,\, f - q_1^2 - \cdots - q_r^2 \, \in \, I \,\,\,
\text{for some} \,q_i \in \R[x_1,\ldots,x_n]_{\leq d} \,\bigr\}. \]
The following chain of inclusions holds:
\begin{equation}
\label{eq:inclusions1}
 {\rm SOS}_1(I) \subseteq {\rm SOS}_2(I) \subseteq {\rm SOS}_3(I) \subseteq
\,\cdots \, \subseteq \, {\rm NN}(I). 
\end{equation}
We now dualize the situation by considering
the subsets of $\R^n$ where the various
$f$ are non-negative. The {\em $d$-th theta body} of the ideal $I$ is the set
\[ \THbody_d(I) \,\,= \,\,
\bigl\{\,x\in \R^n\,|\,f(x)\geq 0 \text{ for all } f \in {\rm SOS}_d(I) \bigr\}. \]
The following reverse chain of inclusions holds among subsets in $\R^n$:
\begin{equation}
\label{eq:inclusions2}
\THbody_1(I) \,\supseteq 
\THbody_2(I) \,\supseteq 
\THbody_3(I) \,\supseteq 
\,\, \cdots \,\,\supseteq \, \overline{\conv(V_\R(I))}.
\end{equation}
This chain of outer approximations can fail to converge in general,
but there are various convergence results when the geometry is nice.
For instance, if the real variety $V_\R(I)$ is compact then
{\em Schm\"udgen's Positivstellensatz}~\cite[\S3]{Sch}
ensures asymptotic convergence. When
$V_\R(I)$ is a finite set, so that
$\conv(V_\R(I))$ is a polytope, then we have finite convergence,
that is, 
$\exists \,d \,: \, \THbody_d(I) = \conv(V_\R(I))$.
This was shown in~\cite{LLR}. For more information on
theta bodies see~\cite{GPT}. The main point we wish to record here is the following:

\begin{thm}{\rm (\cite{GPT, Lasbook})}
Each theta body $\,\THbody_d(I)$ is a spectrahedral shadow.
\end{thm}

\begin{proof}
We may assume, without loss of generality, that
the origin $0$ lies in the interior of $ {\rm conv}(V_\R(I))$.
Then ${\rm SOS}_d(I) $ is the cone over the convex set dual
to $\THbody_d(I)$. Since the class of spectrahedral shadows is closed
under duality, and under intersecting with affine hyperplanes, it suffices
to show that ${\rm SOS}_d(I)$ is a spectrahedral shadow.
But this follows from the formula $\, f - q_1^2 - \cdots - q_r^2 \, \in \, I $,
by an argument similar to that given after (\ref{eq:SOS}).
\end{proof}

In this article we have seen two rather different representations
of the convex hull of a real variety, namely,
the characterization of the algebraic boundary in Section~\ref{sec::convex_hull},
and the representation as a theta body suggested above.
The relationship between these two is not yet well understood.
A specific question is how to 
best compute the algebraic boundary of a 
spectrahedral shadow. This leads to problems in
 elimination theory that seem to be particularly challenging
 for current computer algebra systems.

We conclude by revisiting one of the examples we had seen in Section~\ref{sec::convex_hull}.

\begin{example}{\rm (Example \ref{ex:ellipticcurve} cont.)}
\label{ex:ellipticcurve2}
We revisit the curve
 $X=V(h_1,h_2)$ with
\begin{align*}
h_1 &= x^2+y^2+z^2-1,\\
h_2 &= 19x^2+21y^2+22z^2-20.
\end{align*}
Scheiderer~\cite{Sch} proved that finite convergence holds in
(\ref{eq:inclusions2}) whenever $I$ defines a curve of genus $1$,
such as $X$. We will show that $d=1$ suffices in our example,
i.e.~we will show that
$\THbody_1(I) ={\rm conv}(X)$ for
the ideal $I = \langle h_1,h_2 \rangle$.

We are interested in affine-linear forms $f$ that admit a representation
\begin{equation}
\label{eq::exth1}
f \,\,= \,\,1+ux+vy+wz \,\,\,= \,\,\,  \mu_1 h_1 + \mu_2 h_2\,+\,\sum_i q_i^2.
\end{equation}
Here $\mu_1$ and $\mu_2$ are real parameters.
Moreover, we want $f$ to lie in ${\rm SOS}_1(I)$,
so we require $\deg q_i=1$ for all $i$.
The sum of squares can be written as
\[ \sum_i q_i^2 \,\,\,= \,\,\, (1,x,y,z)\cdot Q \cdot (1,x,y,z)^T,
\qquad
\text{where} \,\, Q \in \SSS_+^4.
\]
After matching coefficients in \eqref{eq::exth1}, we obtain the spectrahedral shadow
\begin{align*}
{\rm SOS}_1(I)
 \,= \,\bigl\{(u,v,w)& \in \R^3 \,\big|\, 
 \exists \mu_1,\mu_2 : \\
 &{\footnotesize \begin{pmatrix}1+\mu_1+20\mu_2 & u & v& w\\u & -\mu_1-19\mu_2 & 0 & 0\\ v & 0 & -\mu_1-21\mu_2 & 0\\w & 0 & 0 & -\mu_1-22\mu_2 \end{pmatrix}\succeq 0} \bigr\}.
\end{align*}
Dual to this is the theta body $\,\THbody_1(I) \, = \,  {\rm SOS}_1(I)^\Delta$.
It has the representation
\[
\THbody_1(I)
\,\, = \,\, \bigl\{(x,y,z) \in \R^3 \,\big|\, \exists u_1,u_2,u_3,u_4 :
\begin{pmatrix}1 & x & y& z\\x & \frac{2}{3}-\frac{1}{3} u_4 & u_1 & u_2\\ y & u_1 & \frac{1}{3}-\frac{2}{3} u_4 & u_3\\z & u_2 & u_3 & u_4 \end{pmatrix}\succeq 0\bigl\}.
\]
We consider the ideal generated by this $4 {\times} 4$-determinant
and its derivatives with respect to $u_1,u_2,u_3,u_4$, we saturate by
the ideal of $3 {\times 3}$-minors, and then we eliminate
$u_1,u_2,u_3,u_4$. The result is the principal ideal $\langle h_4 h_5 h_6 \rangle$,
with $h_i$ as in Example~\ref{ex:ellipticcurve}.
This computation reveals that the algebraic boundary of 
${\rm conv}(X)$ consists of quadrics, and we can conclude that
$\,\THbody_1(I) = {\rm conv}(X)$.

\begin{figure}
\vskip -0.4cm
\includegraphics[width=0.4\textwidth]{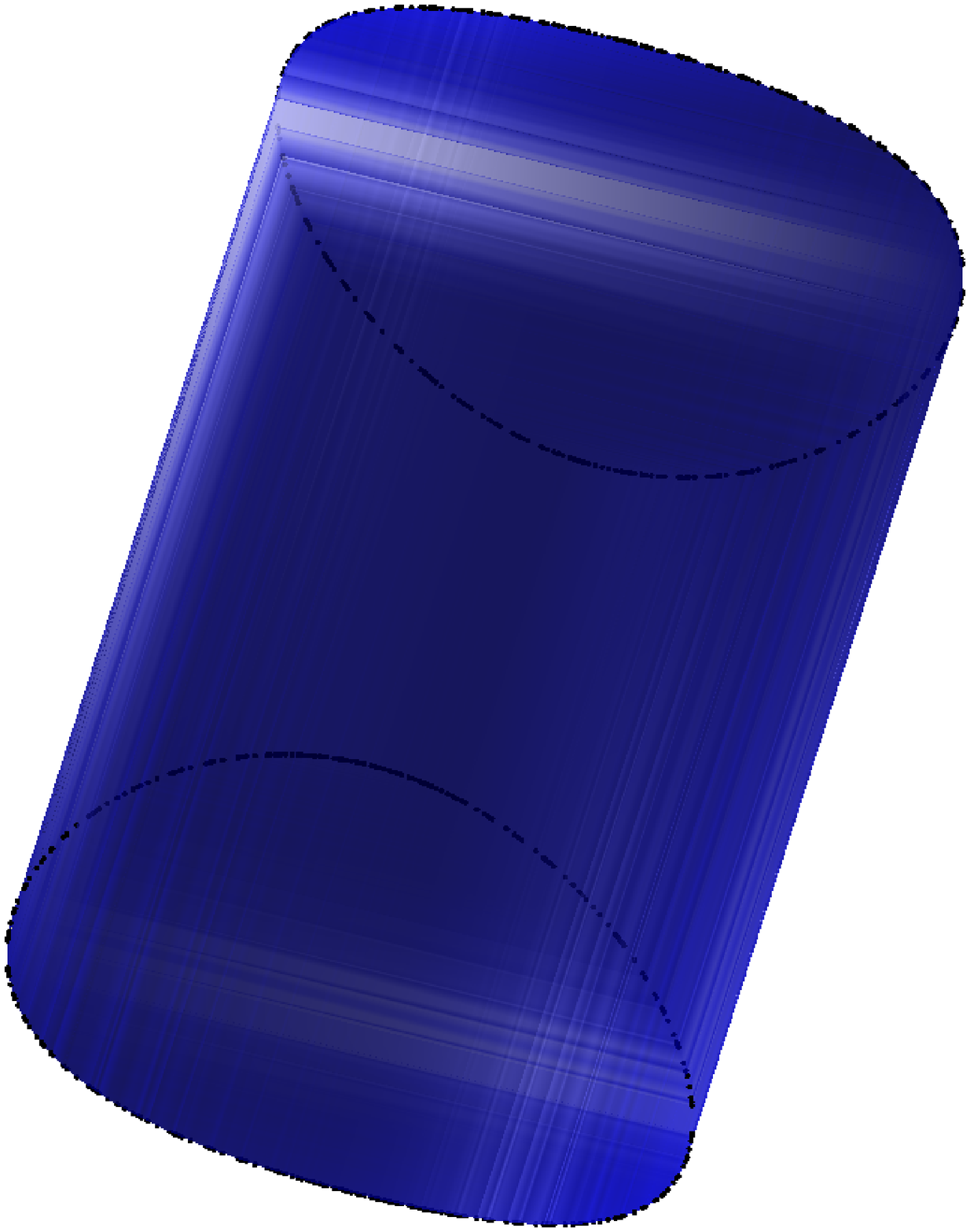}\hfill
\includegraphics[width=0.4\textwidth]{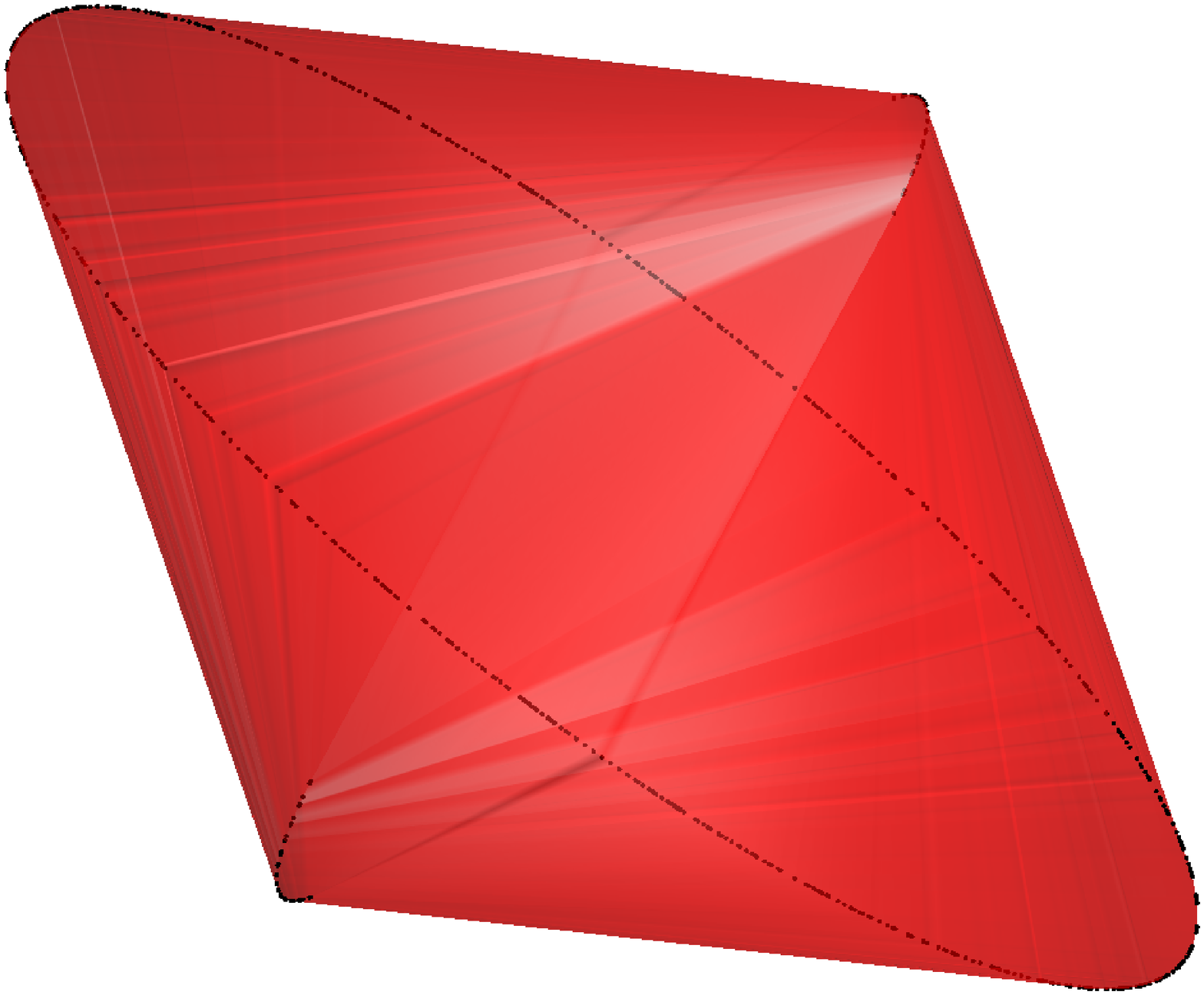}
\caption{
Convex hull of the curve in Figure \ref{fig::lazard_variety} and its dual convex body.
\label{fig::sdp_convhull}}
\end{figure}

Pictures of our convex body and its dual  are shown in
Figure~\ref{fig::sdp_convhull}.
Diagrams such as these can be drawn fairly easily for
any spectrahedral shadow in $\R^3$.
To be precise, the matrix representation of
$\,\THbody_1(I) \,$ and  $\,  {\rm SOS}_1(I)^\Delta\,$
given above can be used to rapidly sample the boundaries 
of these convex bodies, by maximizing many linear functions via
semidefinite programming.
\qed \end{example}
\medskip

\noindent {\bf Acknowledgement.}
This article grew out of three lectures given by Bernd Sturmfels
on March 22-24, 2010,  at the spring school on
Linear Matrix Inequalities and Polynomial Optimization (LMIPO)
at UC San Diego. We thank Jiawang Nie and Bill Helton for organizing that event.
We are grateful to Jo\~ao Gouveia and Raman Sanyal for
their comments on a draft of this paper.
Philipp Rostalski was supported by the Alexander-von-Humboldt Foundation through a 
Feodor Lynen postdoctoral fellowship. Bernd Sturmfels was partially supported by the
U.S.~National Science Foundation through the grants DMS-0456960 and DMS-0757207.

\medskip

\noindent
{\bf Author's address}: Department of Mathematics,
      University of California,
      Berkeley, CA 94720,
      USA. \hfill {\tt \{philipp,bernd\}@math.berkeley.edu}
\end{document}